\documentclass{amsart}

\usepackage[english]{babel}
\usepackage{amsfonts, amsmath, amsthm, amssymb,amscd,indentfirst}

\usepackage{MnSymbol}

\usepackage{tikz-cd}

\usepackage{scalerel}
\usepackage{stackengine,wasysym}

\usepackage{color}
\usepackage{xcolor}
\usepackage{tikz}
\usepackage{esint}

\usepackage{esint}
\usepackage{graphicx}
\usepackage{epstopdf}
\usepackage{amsmath,amssymb,latexsym,indentfirst}
\usepackage{times}
\usepackage{palatino}

\usepackage{stackengine}
\usepackage{scalerel}

\def\R{\mathbb{R}}

\def\d{\partial}

\def\ba{\begin{align}}
	\def\ea{\end{align}}
\def\bp{\begin{proof}}
	\def\ep{\end{proof}}

\theoremstyle{plain}
\newtheorem{theorem}{Theorem}[section]

\newtheorem{proposition}[theorem]{Proposition}
\newtheorem{lemma}[theorem]{Lemma}
\newtheorem{corollary}[theorem]{Corollary}
\theoremstyle{definition}
\newtheorem{definition}[theorem]{Definition}

\newtheorem{conjecture}[theorem]{Conjecture}
\newtheorem{remark}[theorem]{Remark}

\def\R{\mathbb{R}}

\def\d{\partial}

	\def\bp{\begin{proof}}
		\def\ep{\end{proof}}

	\def\R{{\cal R}}

	\def\Lc{{\mathcal L}}

	\def\c{{\mathfrak c}}

	\def\R{\mathbb{R}}
	
	\def\d{\partial}

\begin{document}
		
		\title[Rigidity of static domains in hyperbolic space]{Rigidity of  non-compact static domains in hyperbolic space via positive mass theorems}
		
		\author{S\'{e}rgio Almaraz}
		\address{Universidade Federal Fluminense (UFF) - Instituto de Matem\'{a}tica\\
			Rua M\'{a}rio Santos Braga S/N  24020-140 Niter\'{o}i, RJ, Brazil}
		\email{sergioalmaraz@id.uff.br}
		\author{Levi Lopes de Lima}
		\address{Universidade Federal do Cear\'a (UFC),
			Departamento de Matem\'{a}tica, Campus do Pici, Av. Humberto Monte, s/n, Bloco
			914, 60455-760,
			Fortaleza, CE, Brazil.}
		\email{levi@mat.ufc.br}
		\thanks{
			S. Almaraz has been supported by FAPERJ-202.802/2019, L.L. de Lima has been suported by CNPq  
			312485/2018-2, and both authors have been supported by 
			FUNCAP/CNPq/PRONEX 00068.01.00/15. Part of this article was written while the second author was visiting IHES (Bures-sur-Yvette) and CAMGSD/IST (Lisbon). He would like to thank these institutions for the financial support.	}

		\begin{abstract}
			We single out a notion of staticity  which  applies to any domain in hyperbolic space whose boundary is a non-compact totally umbilical hypersurface. For (time-symmetric) initial data sets modeled at infinity on any of these latter examples, we formulate and prove a positive mass theorem in the spin category under natural dominant energy conditions (both in the interior and along the boundary) whose rigidity statement retrieves, among other things, a sharper version of a recent result by Souam \cite{souam2021mean} to the effect that no such hypersurface admits a compactly supported deformation keeping the original lower bound on the mean curvature. A key ingredient in our approach is the consideration of a family of elliptic boundary conditions on spinors interpolating  between chirality and MIT bag boundary conditions.   
		\end{abstract}

		\maketitle

		\section{Introduction and statements of the rigidity results}\label{intro}

		A {\it{vacuum spacetime}} is a Lorentzian manifold $(\widetilde M^{n+1}, \widetilde
		g)$ with signature $(-,+,...,+)$ satisfying the vacuum Einstein field equation 
		\begin{equation}\label{eq:vacuum}
			{\rm Ric}_{\widetilde g}-\frac{1}{2}R_{\widetilde g}{\widetilde
				g}+\widetilde\Lambda\widetilde g=0,
		\end{equation} 
		where $\widetilde\Lambda\in \R$ is said to be the {\it{cosmological constant}}.
		By taking trace we see that the scalar curvature $R_{\widetilde g}$ is constant
		and hence (\ref{eq:vacuum}) is equivalent to 
		\begin{equation}\label{eq:vacuum:2}
			{\rm Ric}_{\widetilde g}=\Lambda\widetilde g, \quad
			\Lambda=\frac{2}{n-1}\widetilde\Lambda.
		\end{equation}
		We may assume that $\Lambda=\epsilon n$, $\epsilon=0,\pm 1$. Here, we will be
		mainly  interested in the case $\epsilon=-1$, so that
		$\widetilde\Lambda=-n(n-1)/2$, the {\em negative} cosmological constant case.

		If $\widetilde M$ carries a nonzero time-like Killing vector field $X$ such
		that its orthogonal distribution is integrable, then $\widetilde M$ is said to
		be a {\it{static spacetime}}. In this case, if one fixes a space-like slice
		$M\hookrightarrow\widetilde M$ where $X$ never vanishes, then it inherits a Riemannian metric $g$ and one
		can write 
		\begin{equation}\label{eq:metric:spacetime}
			\widetilde g=-V^2dt^2+g, \quad V=\sqrt{-\widetilde g(X,X)}.
		\end{equation}
	  The Einstein
		equation \eqref{eq:vacuum:2} can be written on this slice as
		\begin{equation}\label{eq:static}
			\begin{cases}
				\nabla^2_gV+\Lambda Vg-V{\rm Ric}_g=0,
				\\
				\Delta_gV+\Lambda V=0.
			\end{cases}
		\end{equation}
		
		The discussion above motivates the following classical concept.
		
			\begin{definition}
					A Riemannian manifold $(M,g)$ is said to be  {\it{static}} if there is a
					nontrivial solution $V$ of \eqref{eq:static}, so called a {\it{static
							potential}}. 
		\end{definition}
		
		Static manifolds play a central role in the theory. For instance, they may be
		used as background  spaces when defining mass-type invariants for
		initial data sets in the context of isolated gravitational systems. Roughly
		speaking, one imposes that the given (time-symmetric) initial data set $(M',g')$ approaches at
		infinity a {\em static} manifold $(M,g)$, so that the corresponding mass-type
		invariant somehow extracts the rate of convergence of $g'$ to $g$ as one goes to
		infinity. 	
		One should emphasize
		that
in general the mass invariant so obtained 
		must be thought of as a {linear functional} on the
		space of static potentials of $(M,g)$. We refer to \cite[Section
		2]{de2021conserved} for a discussion of this approach to
		defining conserved quantities in General Relativity.

		The simplest example of a static manifold is $\R^n$ with the canonical flat
		metric $\delta$. A geometric invariant, called the {\it{ADM mass}}, is defined
		for {\em asymptotically flat } manifolds (i.e. manifolds whose geometry at
		infinity approaches, in a suitable sense, that of $(\R^n,\delta$); in this case,
		$\epsilon=0$) and positive mass theorems have been proved in this setting
		\cite{schoen1979proof,witten1981new,bartnik1986mass,schoen2017positive,lohkamp2016higher}.
		More generally, we can consider other static background manifolds, the simplest
		one being 
		the hyperbolic space $\mathbb{H}^n$, which is a static manifold with
		$\epsilon=-1$. Its space of static potentials is identified with the
		($n+1$)-dimensional Minkowski spacetime and positive mass theorems have been
		formulated and proved in this setting as well
		\cite{wang2001mass,chrusciel2003mass,chrusciel2019hyperbolic}. 
		Concerning these contributions recall that,  from the dynamical viewpoint, static manifolds may be regarded as the stationary solutions  of the theory. This naturally leads to 
	the basic question on whether such a manifold may be deformed into another  initial data set satisfying the relevant dominant energy condition and having
	the same asymptotic behavior at infinity. If no such deformation exists we say
	that the given static manifold is {\em rigid}. Since in General Relativity the total energy is
	measured by means of a certain surface integral at spatial infinity, rigidity rules out the
	existence of exotic initial data sets lying in the lowest energy level. From this perspective, the rigidity statements in the positive mass theorems referred to above make sure that $\mathbb R^n$ and $\mathbb H^n$ are rigid in this sense.

		Partially motivated by the so-called AdS/BCFT correspondence
		\cite{takayanagi2011holographic}, a notion of mass for asymptotically hyperbolic
		manifolds with a non-compact boundary, modeled at infinity on the
		half-hyperbolic space $\mathbb H^n_{0}$, which is the bordered non-compact
		manifold obtained by cutting the standard hyperbolic space $\mathbb H^n$ along a
		totally geodesic hypersurface, has been introduced in \cite{almaraz2020mass}.
		In fact, this kind of initial data set may be viewed as an example of a static manifold with  boundary, a notion we isolate in Definition \ref{stat:def:bd} below. Indeed, if our initial data set $(M,g)$ carries a (possibly non-compact) boundary, say $\Sigma$, we argue that it is natural to add  to 
		\eqref{eq:static} the boundary conditions
		\begin{equation}\label{eq:static:bd:cond}
			\begin{cases}
				\pi_g-\lambda \bar g=0,
				\\
				\frac{\d}{\d\eta}V-\lambda V=0,
			\end{cases}
		\end{equation}
		where $\lambda\in\R$, 
		$\pi_g$ is the second fundamental form associated  to the shape operator $\nabla\eta$ of $\Sigma$ with
		respect to the outward unit normal vector  $\eta$, and $\bar g$ is the
		restriction of $g$ to $\Sigma$ 	(in our convention, the shape operator of the boundary of a round unit ball in
		$\R^n$ equals the identity map). As already explained in \cite[Remark 2.2]{almaraz2020mass}, this set of equations arises	 naturally as follows. 
	
		It is known that $\widetilde g$ satisfies \eqref{eq:vacuum} if and only if it is 
	critical for the Einstein-Hilbert functional
	$$
	\widetilde g\mapsto \int_{\widetilde M} (R_{\widetilde g}-2\widetilde\Lambda)\,dv_{\widetilde g},
	$$
	defined on the space of Lorentzian metrics on $\widetilde M^{n+1}$.
	In case $\widetilde M$ has a nonempty boundary $\d\widetilde M$, it is natural to
	consider instead the Gibbons-Hawking-York functional
	\begin{equation}\label{ghy}
	\mathcal F:\widetilde g\mapsto \int_{\widetilde M} (R_{\widetilde g}-2\widetilde\Lambda)dv_{\widetilde g}+
	2\int_{\d \widetilde M} (H_{\widetilde g}-\widetilde\lambda)\,d\sigma_{\widetilde g},
	\end{equation}
	where $\widetilde g$ runs over the space of all Lorentzian metrics on $\widetilde M$ with respect to which $\partial\widetilde M$ is time-like and 
	$H_{\widetilde g}={\rm tr}_{\widetilde g|_{\partial\widetilde M}}\pi_{\widetilde g}$ is the mean curvature of the embedding $\partial\widetilde M\hookrightarrow\widetilde M$. In particular, 
	the geometry of $\d \widetilde M$ now should play a role.  Indeed, critical metrics for $\mathcal F$ are solutions of
	\eqref{eq:vacuum} which additionally satisfy the boundary condition
	\[
	\pi_{\widetilde g}-H_{\widetilde g}\tilde g|_{\d \widetilde M}+\widetilde\lambda\widetilde g|_{\d
		\widetilde M}=0,
	\]
	or equivalently,
	\[
		\pi_{\widetilde g}=\lambda\widetilde g|_{\d \widetilde M}, \quad
	\lambda=\frac{1}{n-1}\widetilde\lambda.
	\]
	 Now, if $X$ is a time-like Killing
	vector field on $\widetilde M$, tangent to $\d \widetilde M$ and with integrable
	orthogonal distribution, we can write $\widetilde g$ in the form
	\eqref{eq:metric:spacetime} and the triple $(M,g,\Sigma)$, where $\Sigma=\partial M=\partial \widetilde M\cap M$, inherits the equations \eqref{eq:static}
	and \eqref{eq:static:bd:cond}. This discussion naturally leads to the following concept, which plays a central role in this work.

	\begin{definition}\label{stat:def:bd}
	 We say that the triple $(M,g,\Sigma)$, $\Sigma=\partial M$, is a
	 {\em static manifold with boundary} with the pair $(\widetilde \Lambda,\widetilde \lambda)$ as {\em cosmological constants}  
	 if there exists $V\not\equiv 0$ such that \eqref{eq:static}
	 and \eqref{eq:static:bd:cond} are satisfied. Any such $V$ is termed a {\em static potential}.  
	\end{definition}

\begin{remark}\label{new:ref}
	A related notion of staticity appears in \cite{cruz2020critical,ho2020deformation} in connection with deformation properties of the scalar and mean curvatures.  
	\end{remark}

	The simplest example of such a manifold is the Euclidean half-space $\mathbb R^n_+$, which is obtained by cutting $\mathbb R^n$ along a hyperplane. In this case, $(\widetilde \Lambda,\widetilde \lambda)=(0,0)$ and the corresponding positive mass theorem is proved in \cite{almaraz2014positive}; see Remark \ref{asymt:flat:gr} below. Next to it we find the hyperbolic half-space $\mathbb H^n_{0}$ mentioned earlier, where $(\widetilde \Lambda,\widetilde \lambda)=(-n(n-1)/2,0)$. 
	A positive mass theorem in this setting has been proved in the spin
		category \cite[Theorem 5.4]{almaraz2020mass}.
		An immediate consequence of the corresponding rigidity statement \cite[Theorem
		1.1]{almaraz2020mass} is the following non-deform\-ability result which has been
		rediscovered in \cite{souam2021mean}. 
		
		\begin{theorem}\label{souam:red}
			A totally geodesic hypersurface in $\mathbb H^n$ can not be compactly deformed
			(as a hypersurface of $\mathbb H^n$) while  keeping it mean convex (that is, with non-negative mean curvature
			everywhere). 
		\end{theorem}
		
		\begin{remark}\label{asymt:flat:gr}
			The corresponding statement for a hyperplane in $\mathbb R^n$ follows
			immediately from the rigidity statement of the positive mass theorem for
			asymptotically flat manifolds with a non-compact boundary proved in
			\cite{almaraz2014positive}. This non-existence of {\em compactly supported}, mean convex deformations of $\mathbb R^{n-1}\hookrightarrow\mathbb R^n$ has also been recovered using an appropriate symmetrization process
			in 
			\cite{gromov2019mean}, which contains an in-depth investigation of the interplay between lower bounds for the scalar curvature (in the interior) and the mean curvature (along the boundary) in more general spaces. {As the action (\ref{ghy}) makes it clear, this fruitful interaction between scalar and mean curvatures already manifests itself in General Relativity}.
		\end{remark}
		
The argument in \cite{souam2021mean} leading to Theorem \ref{souam:red} is quite elementary in the sense that it relies on Aleksandrov's Tangency Principle. 
	In fact, this same reasoning also yields a corresponding rigidity statement for all non-compact {\em totally umbilical} hypersurfaces in $\mathbb H^n$ \cite[Theorem 2]{souam2021mean}. As we shall see in Propositions \ref{secformSigma} and \ref{desc:horo} below, the non-compact domains in $\mathbb H^n$ having  such hypersurfaces as boundaries constitute examples of static manifolds with boundary with cosmological constants of the type $(-n(n-1)/2,\widetilde\lambda)$, for some  $\widetilde\lambda\in[-(n-1),n-1]$. This suggests that, similarly to the contents of Theorem \ref{souam:red} and Remark \ref{asymt:flat:gr} above, the results in \cite{souam2021mean} may alternatively be retrieved as a consequence of more general rigidity statements associated to  appropriate positive mass theorems for initial data sets modeled at infinity on such static domains. The purpose of this paper is precisely to confirm this expectation. 
	Besides placing all the rigidity results mentioned above in their 
 proper conceptual framework, our main contributions (Theorems \ref{main:pmt} and \ref{main:th:horo}) represent a substantial improvement in the sense that their applications to rigidity phenomena comprise not only {purely extrinsic}, compactly supported deformations of the given non-compact boundary but also more general (intrinsic and not necessarily compactly supported)  deformations preserving suitable dominant energy conditions (both in the interior and along the boundary); see Theorem \ref{rig:state} and Remarks \ref{ext:best} and \ref{horo}. 
 Our approach relies on the adoption of a family of elliptic boundary conditions for spinors which somehow interpolates between the well-known chirality and MIT bag boundary conditions (Definition \ref{theta:cond}), as this guarantees, among other things, that the boundary integrals in the Witten-type mass formulas (\ref{maintheo2}) and (\ref{witten:horo}) have the expected shapes.

We now explain our main results in detail. 
		Recall that the so-called {hyperboloid model} for hyperbolic space is given
		by 
		$$
		\mathbb H^n=\{x\in\R^{1,n}\,|\:\langle x,x\rangle_{1,n}=-1\}\subset \R^{1,n},
		$$ 
		where $\R^{1,n}$ is the Minkowski space with the flat Lorentzian metric 
		\[
		\langle
		x,x\rangle_{1,n}=-x_0^2+x_1^2+...+x_n^2,\quad  x=(x_0,x_1,\cdots,x_n)\in\mathbb R^{1,n}.
		\]
		The induced (Riemannian) metric on $\mathbb H^n$ is 
		$$
		b=\frac{dr^2}{1+r^2}+r^2g_{\mathbb S^{n-1}},
		$$
		where $g_{\mathbb S^{n-1}}$ stands for the round metric on the unit sphere
		$\mathbb S^{n-1}$, $r=|x'|_\delta$, $x'=(0,x_1,\cdots,x_n)$ and  
		\[
		|x'|_\delta=\sqrt{x_1^2+...+x_n^2}.
		\] 
		It is immediate that $(\mathbb H^n,b)$ is a complete static manifold
		with $\widetilde\Lambda=-n(n-1)/2$ whose space of static potentials is given by  $\mathcal N_b=[V_{(0)},V_{(1)},\cdots,V_{(n)}]$, where $V_{(i)}=x_i|_{\mathbb H^n}$.

		For each $s\in \mathbb R$ consider $\mathbb H^n_{s}=\{x\in\mathbb H^n;x_1\leq
		s\}$. We will see in next proposition that $\Sigma_s=V_{(1)}^{-1}(s)$, the
		boundary of $\mathbb H^n_s$, is a totally umbilical hypersurface of $\mathbb H^n$.
		In fact, the family $\{\Sigma_s\}_{s\neq 0}$ constitutes the {\em equidistant
			hypersurfaces} of $\Sigma_0$, which is totally geodesic. 
		Notice that the outward pointing unit vector field along $\Sigma_s$ is
		$\eta_s=\nabla_bV_{(1)}/|\nabla_b V_{(1)}|$.

		\begin{proposition}\label{secformSigma}
			Along $\Sigma_s$  we have $|\nabla_{b}V_{(1)}|=\sqrt{1+s^2}$. Also, the second
			fundamental form $\Pi_s$ of $\Sigma_s$ is given by
			\[
				\Pi_s=\lambda_s\gamma_s,\quad \lambda_s=\frac{V_{(1)}}{|\nabla_b
					V_{(1)}|}=\frac{s}{\sqrt{1+s^2}}. 
			\]
			where $\gamma_s=b|_{\Sigma_s}$ is the induced metric. Moreover, for each $i\neq
			1$,
			\begin{equation}\label{neumanncond}
				\frac{\partial V_{(i)}}{\partial\eta_s}=\lambda_s V_{(i)}.
			\end{equation}
		In particular, the triple $(\mathbb H^n_s,\Sigma_s,b)$ is a static manifold with boundary satisfying 
		\[
		(\widetilde\Lambda,\widetilde\lambda)=(-\frac{n(n-1)}{2},(n-1)\lambda_s), 
		\]
	and the corresponding space of static potentials is 
	\[
	\mathcal N_{b,s}=\left[V_{(0)},V_{(2)},\cdots, V_{(n)}\right].
	\]
		\end{proposition}
		
		\begin{proof}
			A direct computation shows that 
			$$
			\nabla_bV_{(i)}  = \sum_j\langle \partial_{x_i},\partial_{x_j}\rangle_{1,n} \partial_{x_j} +x_ix,
			$$
			where $\{\partial_{x_i}\}_{i=0}^n$ is the standard orthonormal basis. It follows that 
			\begin{equation}\label{formgrad}
				\langle \nabla_b V_{(i)},\nabla_b V_{(j)}\rangle=\langle \partial_{x_i},\partial_{x_j}\rangle_{1,n}+x_ix_j,
					\end{equation}
			so that $|\nabla_b  V_{(1)}|=\sqrt{1+s^2}$ along $\Sigma_s$. Thus, if $X,Y$ are vector fields 
			tangent to $\Sigma_s$,  
			\begin{eqnarray*}
				\Pi_{s}(X,Y) & = & -\langle (\nabla_b)_{X}Y,\eta_s\rangle\\
				& = & -(1+s^2)^{-1/2}\langle (\nabla_b)_XY,\nabla_b V_{(1)}\rangle\\
				& = & (1+s^2)^{-1/2}\langle Y,(\nabla_b)_X\nabla_b V_{(1)}\rangle \\
				& = & (1+s^2)^{-1/2}(\nabla_b^2V_{(1)})(X,Y).
			\end{eqnarray*}
		We now remark that the 
			staticity equation (\ref{eq:static}) implies 
			\[
				\nabla_b^2V=Vb, \quad V\in\mathcal N_b, 
			\]	
				 so we get
			\[
			\Pi_s(X,Y)=s(1+s^2)^{-1/2}\gamma_s(X,Y),
			\]
			as desired.
			 Finally, note that (\ref{neumanncond})
			is equivalent to the identity
			\[
			\langle \nabla_b V_{(i)}, \nabla_b V_{(1)}\rangle=V_{(i)}V_{(1)}, \quad i\neq
			1,
			\]
			which is immediate from (\ref{formgrad}). 
		\end{proof}

		We now define the concept of an asymptotically hyperbolic manifold with
		non-compact boundary having $(\mathbb H^n_s,b,\Sigma_s)$ as an asymptotic model.
		Recall that one can  parameterize $\mathbb H^n$ by polar coordinates
		$(r,\varphi)$, where $r=|x'|_\delta$ as before and
		$\varphi=(\varphi_1,...,\varphi_{n-1})\in \mathbb S^{n-1}$. 
		Let $\{\partial_{\varphi_1},...,\partial_{\varphi_{n-1}}\}$ be an orthonormal frame
		for $g_{\mathbb S^{n-1}}$. Then $\{\mathfrak f_ i\}_{i=1}^{n}$, with
		$\mathfrak f_{a}=r^{-1}\d_{\varphi_{a}}$, $a=1,...,n-1$, and $\mathfrak f_n=\sqrt{r^2+1}\d_r$ is an
		orthonormal frame for $b$.
		Given $s$ as above, for all $r_0$ large enough let us set $\mathbb
		H^n_{s,r_0}=\{x\in\mathbb H^n_s;r(x)\geq r_0\}$.

		\begin{definition}\label{def:as:hyp}
			We say that $(M^n,g, \Sigma)$ is an $s$-{{\em asymptotically hyperbolic}} (briefly, $s$-{\em AH}) manifold if there exist a region
			$M_{\text{ext}}\subset M$ and a diffeomorphism (a chart at infinity)
			\[
				F:\mathbb H_{s,r_0}^n\to  M_{\text{ext}},
			\]
			for some $r_0>0$, such that the induced metric $F^*g$ on $\mathbb
			H^n_{s,r_0}$ satisfies the asymptotic conditions 
			\begin{equation}\label{asympthyp}
				|F^*g-b|_b+\sum_{i=1}^n|\mathfrak f_i(F^*g)|_b+\sum_{i,j=1}^n|\mathfrak f_i\mathfrak f_j(F^*g)|_b=O(r^{-\sigma}), \quad \sigma>n/2.
			\end{equation}
			{We further assume that $r(R_g+n(n-1))\in
			L^1(M)$ and $r(H_g-(n-1)\lambda_s)\in L^1(\Sigma)$, where the radial function $r$ has been smoothly extended to the whole of $M$.}
		\end{definition}

		In the next section we define a notion of mass for this kind of asymptotically hyperbolic manifold with a non-compact boundary. In case the underlying manifold is spin, we will be able to establish the corresponding positive mass theorem under suitable dominant energy conditions. This is the content of our main result, Theorem \ref{main:pmt} below. As a consequence of the corresponding rigidity statement, the following result is easily obtained.

\begin{theorem}\label{rig:state}
Let $(M^n,g, \Sigma)$ be an $s$-AH spin manifold with $R_g\geq -n(n-1)$ and $H_g\geq (n-1)\lambda_s$. Assume further that (\ref{asympthyp}) holds with $\sigma>n$. Then $(M^n,g,\Sigma)=(\mathbb H^n_s,b,\Sigma_s)$ isometrically.
\end{theorem}
\begin{proof}
			Since any $V\in\mathcal N_{b,s}$ satisfies $V=O(r)$ as $r\to+\infty$, it is immediate that the decay assumption (\ref{asympthyp}) with $\sigma>n$ implies that the mass vector $P_s(F)$ in (\ref{mass:vector}) vanishes. The result then follows from Theorem \ref{main:pmt} and Remark \ref{rig:appl}.
		\end{proof}
		
Clearly, this result means that the static manifold with boundary $(\mathbb H^n_s,b,\Sigma_s)$ is rigid in the sense discussed previously. 
		
\begin{corollary}\label{souam:cor}\cite[Theorem 2]{souam2021mean}
	The embedding 
		$\Sigma_s\hookrightarrow\mathbb H^n$ can not be compactly  deformed (as a hypersurface of $\mathbb H^n$) while  keeping its mean curvature at least $(n-1)\lambda_s$ everywhere. 	
\end{corollary}
		
	\begin{remark}\label{ext:best}
	Theorem \ref{rig:state} actually implies a sharper version of Corollary \ref{souam:cor} in the sense that the deformation does {\em not} have to be compactly supported. This follows from the fact that,  in the specific case of deformations of the embedding $\Sigma_s\hookrightarrow\mathbb H^n$, the assumption $\sigma>n$  may be rephrased in terms of suitable decay rates for the fundamental forms of the deformation viewed as a graph over $\Sigma_s$ in a neighborhood of infinity.
	\end{remark}

		\begin{remark}\label{mari}
		It is well-known that the condition $R_g\geq -n(n-1)$ may be interpreted as a dominant energy condition  in the interior of $M$. Similarly, it turns out that the condition $H_g\geq (n-1)\lambda_s$ may also be viewed as a dominant energy condition along $\Sigma$ in the spirit of \cite{almaraz2021spacetime}, which treats the case $s=0$.
	\end{remark}	
		
		\begin{remark}\label{horo}
			As stated above, Theorem \ref{rig:state} and Corollary \ref{souam:cor} do not contemplate the case of a non-compact (static) domain in $\mathbb H^n$ whose boundary is a horosphere, a situation that, at least for compactly supported deformations, has also been considered in \cite{souam2021mean}. It turns out that this case may be approached by means of a somewhat involved variation of the methods leading to Theorem \ref{rig:state}. More precisely, it is possible to formulate and prove a positive mass theorem for asymptotically hyperbolic spin manifolds modeled at infinity on such static domains. As expected, its  rigidity statement retrieves a sharper version (in the sense of Remark \ref{ext:best} above) of the corresponding result in \cite[Theorem 2]{souam2021mean}; {see Section \ref{horo:case} for details.}
			\end{remark}
		
		This paper is organized as follows. In Section \ref{geomass}, we introduce the concept of mass for $s$-AH manifolds and establish its geometric invariance. The corresponding positive mass theorem, including its rigidity statement, is proved in Section \ref{proof:main} under the spin assumption. This uses in a crucial way the $\theta$-boundary conditions on spinors, whose properties are discussed in Section \ref{spinorsbd}. Finally, in Section \ref{horo:case} we indicate how our main theorem may be extended to the horospherical case.

		
		\section{The geometric invariance of the mass functional}\label{geomass}
		
		Here we define a mass-type invariant for an $s$-AH manifold $(M,g,\Sigma)$
		as in Definition \ref{def:as:hyp} and study its invariance properties under the group of isometries of the model space. 
		The argument here is quite similar to that appearing in \cite[Section 3]{almaraz2020mass}, so we omit some details.  
		
		Since ultimately the mass will  depend only on the asymptotic geometry of the manifold, we may appeal to the identification provided by the chart $F$ to work in  $\mathbb H^n_{s,r_0}$. We set, for $r_0<r_1<r_2$,
		$$
		A_{r_1,r_2}=\{x\in \mathbb H^n_{s,r_0};\:r_1\leq |x'|_\delta \leq r_2\},\:\:\:\Sigma_{r_1,r_2}=\{x\in
		\mathbb H^n_{s,r_0}\cap \Sigma_s;\:r_1\leq |x'|_\delta \leq r_2\},
		$$ 
		and
		$S^{n-1}_{r,+}=\{x\in \mathbb H^n_{s,r_0};\:|x'|_\delta=r\}$,
		so that
		\[
		\d A_{r_1,r_2}=S^{n-1}_{r_1,+}\cup \Sigma_{r_1,r_2}\cup S^{n-1}_{r_2,+}.
		\]
		In other words, $S^{n-1}_{r,+}$ is the portion of the geodesic sphere centered at the `'origin'' $(1,0,\cdots,0)$ in $\mathbb R^{1,n}$ and with radius $\sinh^{-1}r$ lying inside $\mathbb H^n_{s,r_0}$.
		We represent by $\mu$ the outward unit normal vector field to $S^{n-1}_{r_1,+}$ or
		$S^{n-1}_{r_2,+}$, with respect to the metric $b$. Also, we set $S^{n-2}_r=\partial
		S_{r,+}^{n-1}\hookrightarrow \Sigma_s$, oriented by its outward unit conormal field $\vartheta$, again 
		with respect to $b$.
		Finally, we set ${e}=g-b$ and
		define the $1$-form
		\begin{equation}\label{charge}
			\mathbb U(V,e)=V({\rm div}_be-d{\rm tr}_be)-{\nabla_bV}\righthalfcup e+{\rm
				tr}_be\, dV.
		\end{equation}
		
		\begin{theorem}\label{finitemass}
			If $(M,g,\Sigma)$ is an $s$-AH manifold then the quantity
			\begin{equation}\label{massdef}
				\mathfrak m_{s,F}(V)=\lim_{r\to
					+\infty}\left[\int_{S^{n-1}_{r,+}}\langle\mathbb U(V,e),\mu\rangle
				dS^{n-1}_{r,+}-
				\int_{S^{n-2}_{r}}Ve(\eta_s,\vartheta)dS^{n-2}_{r}\right], \quad V\in\mathcal N_{b,s},
			\end{equation}
			exists and is finite.	
		\end{theorem}
	
\begin{proof}
		As in the proof of \cite[Theorem 3.1]{almaraz2020mass}, the argument is based on the expansions of both the scalar curvature and the mean curvature around the
			background metric $b$. To begin with, a well-known computation gives,
for any $V\in \mathcal N_{b,s}$,  
			\[
				V(R_g+n(n-1))={\rm div}_b\mathbb U(V,e)+Q_b(e,V), 
			\]
			where $Q_b(e,V)$  is linear in $V$ and at least quadratic in $e$.
			We now perform the integration of this identity over the half-annular region
			$A_{r,r'}\subset \mathbb H^n_{s,r_0}$ and explore the imposed boundary conditions on the underlying static domain, namely, 
			\begin{equation}\label{Vbd}
				\frac{\partial V}{\partial\eta_s}=\lambda_s V, \quad \Pi_s=\lambda_s \gamma_s;
			\end{equation}
		compare with Proposition \ref{secformSigma}.
Using the well-known first variation formula for the mean curvature of the boundary  
		 we obtain
			\begin{eqnarray*}
					\mathcal F_{r_1,r_2}(g)& := & \int_{A_{r_1,r_2}}V(R_g+n(n-1))dM_b+2\int_{\Sigma_{r_1,r_2}}V(H_g-(n-1)\lambda_s)d\Sigma_{\gamma_s}\\
				& \approx  & 	\int_{S^{n-1}_{r_2,+}}\langle\mathbb U(V,e),\mu\rangle dS^{n-1}_{r_2,+}-
				\int_{S^{n-1}_{r_1,+}}\langle\mathbb U(V,e),\mu\rangle dS^{n-1}_{r_1,+}\\
				& & -\int_{\Sigma_{r_1,r_2}}({\nabla_bV}\righthalfcup
				e)(\eta_s)d\Sigma_{\gamma_s}+\int_{\Sigma_{r,r'}}{\rm tr}_be\, dV(\eta_s)d\Sigma_{\gamma_s}\\
				& & -\int_{\Sigma_{r_1,r_2}}V{\rm div}_{\gamma_s}Xd\Sigma_{\gamma_s}
				-\int_{\Sigma_{r_1,r_2}}V\langle\Pi_b,e\rangle_\gamma d\Sigma_{\gamma_s}
			\end{eqnarray*}
		where 	$X$ is the vector field dual to the $1$-form $(\eta_s,\cdot)|_{T\Sigma}$  the symbol $\approx$ means that we are discarding certain integrals over $\Sigma_{r_1,r_2}$ or $A_{r_1,r_2}$ which vanish as $r_1\to+\infty$ due to the assumption $\sigma>n/2$. Also, we are omitting the restriction symbol on 
		$e|_\Sigma$.  
			Now observe that (\ref{Vbd}) leads to 
			\[
			{\rm tr}_be\, dV(\eta_s)=V\langle\Pi_b,e\rangle_{\gamma_s},
			\]
			so that, after a little manipulation,
			we end up with 
			\begin{eqnarray*}
				\mathcal F_{r_1,r_2}(g)
				& \approx  & 	\int_{S^{n-1}_{r_2,+}}\langle\mathbb U(V,e),\mu\rangle dS^{n-1}_{r_2,+}-
				\int_{S^{n-1}_{r_1,+}}\langle\mathbb U(V,e),\mu\rangle dS^{n-1}_{r_1,+}\\
				& & -\int_{\Sigma_{r_1,r_2}}{\rm div}_{\gamma_s}(VX)d\Sigma_{\gamma_s}\\
				& =  & 	\int_{S^{n-1}_{r_2,+}}\langle\mathbb U(V,e),\mu\rangle dS^{n-1}_{r_2,+}-
				\int_{S^{n-2}_{r_2}}Ve(\eta_s,\vartheta) dS^{n-2}_{r_2}\\
				& & -\left(\int_{S^{n-1}_{r,+}}\langle\mathbb U(V,e),\mu\rangle dS^{n-1}_{r,+}-
				\int_{S^{n-2}_{r_1}}Ve(\eta_s,\vartheta) dS^{n-2}_{r_1}\right).
			\end{eqnarray*}
			 Bearing in mind that $V=O(r)$, the integrability assumptions on $r(R_g+n(n-1))$
			and $r(H_g-(n-1)\lambda_s)$ in Definition \ref{def:as:hyp} clearly imply that $\mathcal F_{r_1,r_2}(g)\to 0$ as $r_1\to+\infty$, which completes the proof.  
		\end{proof}
		
		We must think of $\mathfrak m_{s,F}$ as a 
		linear functional on the space $	\mathcal N_{b,s}$
		of static potentials satisfying the given boundary conditions.
		We must be aware, however, that the decomposition $g=b+e$ used above depends on the choice
		of a chart at infinity (namely, the diffeomorphism $F$), so  we
		need to check that $\mathfrak m_{s,F}$ behaves as expected when we pass from one such
		chart to another. For this we need a couple of results whose proofs follow from well-known principles \cite{chrusciel2003mass,almaraz2020mass}.  
		
		\begin{lemma}\label{exact}
			If $V\in\mathcal N_{b,s}$ and $X$ is a vector field then 
			\begin{equation}\label{exact1}
				\mathbb U(V,\mathcal L_Xb)={\rm div}_b\mathbb V(V,X,b),
			\end{equation}
			where the $2$-form is explicitly given by 
			\begin{equation}\label{exact2}
				\mathbb V_{ik}=V(X_{i;k}-X_{k;i})+2(X_kV_i-X_iV_k).
			\end{equation}
		Here, the semicolon denotes covariant differentiation with respect to $b$.
		\end{lemma}

		\begin{lemma}\label{rigid}
			If $F:\mathbb H^n_s\to \mathbb H^n_s$ is a diffeomorphism such that
			$F^*b=b+O(r^{-\sigma})$ then there exists an isometry $A$ of $\mathbb H^n_s$ such that
			\[
			F=A+O(r^{-\sigma}).
			\] 
		\end{lemma}
		
		\begin{remark}\label{iso:desc}
			In the hyperboloid model $\mathbb H^n\hookrightarrow\mathbb R^{1,n}$, the isometry group of hyperbolic $n$-space gets identified to $O^\uparrow(1,n)$, the subgroup of linear isometries of $(\mathbb R^{1,n},\langle\,,\rangle_{1,n})$ preserving time orientation.  It is clear that any $A\in O^\uparrow(1,n)$ preserving $\Sigma_s$, and hence defining an isometry of $\mathbb H^n_s$, also preserves $\Sigma_0$, which is an isometric copy of hyperbolic $(n-1)$-space. Thus, the group of isometries of $\mathbb H^n_s$, which appears in Lemma \ref{rigid} above, may be identified to $O^\uparrow(1,n-1)$. We thus obtain a natural representation $\rho^s$ of $O^\uparrow(1,n-1)$ on $\mathcal N_{b,s}$ by setting $\rho^s_A(V)=V\circ A^{-1}$, which is easily shown to be {irreducible}. 
			\end{remark}

		Suppose now that we have two diffeomorphisms, say $F_{1},F_2:\mathbb
		H^n_{s,r_0}\to M_{\rm ext}$, defining charts at infinity as above and consider
		$F=F^{-1}_1\circ F_2:\mathbb
		H^n_{s,r_0}\to \mathbb
		H^n_{s,r_0}$. It is clear that
		$F^*b=b+O(r^{-\sigma})$, $\sigma>n/2$, so by Lemma \ref{rigid},
		$F=A+O(r^{-\sigma})$ for some isometry $A$. The next result establishes the
		geometric invariance of the mass-type invariant appearing in Theorem
		\ref{finitemass}.
		
		\begin{theorem}\label{geoinv}
			Under the conditions above, there holds
			\[
			\mathfrak m_{s,F_1}(V)=\mathfrak m_{s,F_2}(\rho^s_A(V)), \quad V\in\mathcal N_{b,s}.
			\]
		\end{theorem} 
		
		\begin{proof}
			As in the proof of \cite[Theorem 3.4]{almaraz2020mass}, we may assume that $A$ is the identity, so that $F={\rm exp}\circ \zeta$,
			where $\zeta$ is a vector field on $\mathbb
			H^n_{s,r_0}$ which vanishes at infinity and is
			tangent to $\Sigma_s$ everywhere. Now set
			\[
			e_1=g_1-b,\quad g_1=F_1^*g,
			\quad
			e_2=F^*_2g-b=F^*g_1-b,
			\]
			so that 
			\[
			e:=e_2-e_1=\Lc_\zeta b+R_1,
			\]
			where $\Lc$ is Lie derivative and $R_1$ is a certain remainder. It follows that 
			\[
			\mathbb U(V,e)=\mathbb U(V,\Lc_\zeta b)+R_2,
			\]
			where $R_2$ is a remainder that vanishes, as $r\to +\infty$, after integration over $S^{n-1}_{r,+}$. Hence, 
			\begin{eqnarray*}
				\lim_{r\to+\infty}\int_{S^{n-1}_{r,+}}\langle\mathbb{U}(V,e),\mu\rangle
				dS^{n-1}_{r,+} 
				& = &\lim_{r\to+\infty}\int_{S^{n-1}_{r',+}}\langle\mathbb{U}(V,\Lc_\zeta
				b),\mu\rangle dS^{n-1}_{r,+}\\
				& = & 
				-\lim_{r\to+\infty}\int_{\Sigma_{(r)}}\langle\mathbb{U}(V,\Lc_\zeta b),\eta_s\rangle
				d\Sigma_s,
			\end{eqnarray*}
			where we used (\ref{exact1}) in the last step to transfer the integral to $\Sigma_{(r)}$, the compact domain of $\Sigma_s$ enclosed by $S^{n-2}_r$.
			We fix an adapted orthonormal frame $\{\mathfrak e_i\}_{i=1}^n$ so that $\mathfrak e_1=-\eta_s$ is the inward unit normal to $\Sigma_s$, so that $({\Pi_s})_{\alpha\beta}=\Gamma_{\alpha\beta}^1=-\Gamma^\alpha_{1\beta}$, $2\leq\alpha,\beta\leq n$. By  (\ref{exact1}) with $X=\zeta$,
			\[
			\langle\mathbb{U}(V,\Lc_\zeta b),\eta_s\rangle=b^{jk}\mathbb
			V_{ij;k}\eta_s^i={-}\mathbb V_{1k;k}={-}\mathbb V_{1\alpha;\alpha}=
			{\rm div}_\gamma (\eta_s\righthalfcup\mathbb V),
			\]
		and hence,
			\[
			\int_{\Sigma_r}\langle\mathbb{U}(V,\Lc_\zeta b),\eta_s\rangle
			d\Sigma_r=\int_{\Sigma_r}\mathbb
			V(\eta_s,\vartheta)d\Sigma_r={-}\int_{S^{n-2}_{r}}\mathbb V_{1\alpha}\vartheta^\alpha
			dS^{n-2}_r.
			\]
			We now observe that (\ref{Vbd}) may be used to check that $V_1=-\lambda_s V$ and 
			\[
			\zeta_{1;\alpha} = \zeta_{1,\alpha}+\Gamma_{\alpha i}^1\zeta_i={\Pi_s}_{\alpha\beta}\zeta_\beta=\lambda_s\zeta_\alpha,
			\]
			where we used that  $\zeta_1=0$ (because $\zeta$ is tangent to $\Sigma$). Similarly, $\zeta_{\alpha;1}=\zeta_{\alpha,1}{-}\lambda_s\zeta_{\alpha}$. Thus, 
			 	\[
			 \mathbb V_{1\alpha}=V(\zeta_{1;\alpha}-\zeta_{\alpha;1})+2(\zeta_\alpha
			 V_1-\zeta_1V_\alpha)={-V\zeta_{\alpha,1}},
			 \]
			 so that
			\begin{equation}\label{lasti}
			\lim_{r\to+\infty}	\int_{S^{n-1}_{r,+}}\langle\mathbb{U}(V,e),\mu\rangle
				dS^{n-1}_{r,+}=-\lim_{r\to+\infty}\int_{S^{n-2}_{r}}V\zeta_{\alpha,1}\vartheta^\alpha
				dS^{n-2}_r.
			\end{equation}  
			The argument is completed by noticing that 
			\[
			\left(-\int_{S^{n-2}_{r}}Ve_2(\eta_s,\vartheta)dS^{n-2}_{r}\right)-\left(-\int_{S^{n-2}_{r}}Ve_1(\eta_s,\vartheta)dS^{n-2}_{r}\right)=-\int_{S^{n-2}_{r}}
			Ve(\eta_s,\vartheta)dS^{n-2}_{r},
			\]
			and this equals
			\[
			-\int_{S^{n-2}_{r}}V(\Lc_\zeta
			b)(\eta_s,\vartheta)dS^{n-2}_{r}-\int_{S^{n-2}_{r}}VR_1(\eta_s,\vartheta)dS^{n-2}_{r},
			\]
			where the last integral vanishes at infinity. Finally, the remaining integral
			may be evaluated as 
			\begin{eqnarray*}
				-\int_{S^{n-2}_{r}}V(\Lc_\zeta b)(\eta_s,\vartheta)dS^{n-2}_{r}
				& = & \int_{S^{n-2}_{r}}V(\zeta_{\alpha;1}+\zeta_{1;\alpha})\vartheta^\alpha
				dS^{n-2}_{r}\\
				& = &
				\int_{S^{n-2}_{r}}V\zeta_{\alpha,1}\vartheta^\alpha
				dS^{n-2}_r,
			\end{eqnarray*}
			which  cancels out the right-hand side of (\ref{lasti}) as $r\to+\infty$.
		\end{proof}

		We now explore the consequences of Theorem \ref{geoinv}.  For this we  introduce a `'Lorentzian'' inner product $\langle\,,\rangle_s$ on $\mathcal N_{b,s}$
		by declaring that 	$\{V_{(0)}, V_{(2)},\cdots,V_{(n)}\}$ is an orthonormal basis  with $\langle V_{(0)},V_{(0)}\rangle_s=1$ and  $\langle V_{(j)},V_{(j)}\rangle_s=-1$, $j\geq 2$. Thus, we agree that $V\in \mathcal N_{b,s}$ is {\em future-directed} if $\langle V,V_{(0)}\rangle_s>0$. 
		The key point now is that the
		isometry group $O^\uparrow(1,n-1)$ of the background static space $(\mathbb
		H^n_{s},b,\Sigma_s)$ acts {\em isometrically} on $(\mathcal N_{b,s},\langle\,,\rangle_s)$ by means of the representation $\rho^s$ considered in Remark \ref{iso:desc}. {Since $\rho^s$ is irreducible,  all vectors in $\mathcal N_{b,s}$ should equally contribute to defining a single vector-valued mass invariant.
		Precisely, if for
		any  chart at infinity $F$ as above we set
		\begin{equation}\label{mass:vector}
		P_{s}(F)_a=\mathfrak m_{s,F}(V_{(a)}),\quad a=0,2,\cdots,n, 
		\end{equation}
		then Theorem \ref{geoinv}
		guarantees that the causal properties of $P_s(F)$ (e.g, whether it is space-like, isotropic or time-like, its past/future-directed nature in the two latter cases, its  Lorentzian length with respect to $\langle\,,\rangle_s$, etc.) are {\em chart
			independent} indeed. This
		suggests the following conjecture.
		
		\begin{conjecture}\label{conjmp}
			Let $(M,g,\Sigma)$ be an $s$-AH manifold as in Theorem \ref{finitemass} above with $R_g\geq -n(n-1)$ and $H_g\geq (n-1)\lambda_s$. Then for 
			any chart at infinity $F$ the vector $P_s(F)$ is time-like and  future-directed
			unless it vanishes, in which case $(M,g,\Sigma)$ is isometric to $(\mathbb
			H^n_s,b,\Sigma_s)$.
		\end{conjecture}
		
		\begin{remark}\label{rig:appl}
		Whenever this conjecture holds true we may define the numerical invariant
		\[
		\mathfrak m_s:=\sqrt{\langle P_s(F),P_s(F)\rangle_s}=\sqrt{P_s(F)_0^2-\sum_{a=2}^nP_s(F)_a^2},
		\]
		which happens to be independent of the chosen chart. This may be regarded as the total mass of the isolated 
		system whose (time-symmetric) initial data set is $(M,g,\Sigma)$. Notice that
		$\mathfrak m_s\geq 0$ with the equality holding if and only if
		$(M,g,\Sigma)$ is isometric to $(\mathbb H^n_s,b,\Sigma_s)$.
		\end{remark} 
	
As remarked in the
		introduction, our main result here confirms the
		conjecture (and hence the physical interpretation for $\mathfrak m_s$ above) in case $M$ is spin.
		
		\begin{theorem}\label{main:pmt}
			Conjecture \ref{conjmp} holds true in case $M$ is spin.
			\end{theorem}
		
		As already observed, Theorem \ref{rig:state} (and hence Corollary \ref{souam:cor}) is an immediate consequence of Theorem \ref{main:pmt}.

		
		\section{Spinors on manifolds with boundary}\label{spinorsbd}
		
		In this section we review the results in the theory of spinors on  manifolds carrying a (possibly non-compact) boundary $\Sigma$ which are needed in the rest of the paper, including the appropriate integral  version of the celebrated Lichnerowicz formula. Also, we introduce a family of boundary conditions for spinors which interpolates between chirality and MIT bag boundary conditions (Definition \ref{theta:cond} below) and plays a central role in the proof of our main results.   
		
		\subsection{The integral Lichnerowicz formula on spin manifolds with boundary}\label{lichbd}
		We assume that the given manifold $(M,g)$ is spin and fix once and for all a spin structure on  $TM$. We denote by $\mathbb SM$ the associated (Hermitian) spinor bundle and by $\nabla$ both the Levi-Civita connection of $TM$
		and its compatible lift to $\mathbb SM$. 
		Also, each tangent vector $X$ induces a linear map
		$\c(X):\mathbb SM\to\mathbb SM$, the  (left) {Clifford multiplication} by $X$. For our purposes, it suffices to know that these structures satisfy a few compatibility conditions: 
		
			\begin{enumerate}
			\item $\mathfrak c(X)\mathfrak c(Y)+\mathfrak c(Y)\mathfrak c(X)=-2\langle X,Y\rangle_g$;
			\item $\langle \mathfrak c(X)\Psi,\mathfrak c(X)\Phi\rangle=|X|^2\langle\Psi,\Phi\rangle$;
			\item $\nabla_X(\mathfrak c(Y)\Psi)=\mathfrak c(\nabla_XY)\Psi+\mathfrak c(Y)(\nabla_X\Psi)$.
		\end{enumerate} 
		Here, $X,Y\in\Gamma(TM)$ are tangent vector fields and $\Psi,\Phi\in\Gamma(\mathbb SM)$ are spinors. 	When emphasizing the dependence of $\c$ on $g$ is needed, we append a sub/superscript and  write $\c=\c^g$ for instance (and similarly for other geometric invariants associated to  the given spin structure).
		
		We define the corresponding Dirac operator acting on spinors by 
		\[
		D=\c\circ\nabla,
		\]
		where we  view the connection as a linear map 
		\[
		\nabla:\Gamma(\mathbb SM)\to \Gamma(T^*M\otimes\mathbb SM)=\Gamma(TM\otimes\mathbb SM)
		\] 
		and the identification $T^*M=TM$ comes 
		from the metric. 
	For our purposes, it is convenient to slightly modify this classical construction. Thus, 
		we define the {\em Killing connections} on  $\mathbb S M$ by
		\begin{equation}\label{kill:conn}
		\nabla_X^{\pm}=\nabla_X\pm\frac{\bf i}{2}\c(X),\quad X\in\Gamma(TM),
		\end{equation}
	so that 
		the corresponding {\em Killing-Dirac operators} are defined in the usual way, namely,
		\[
		D^{\pm}:=\c\circ\nabla^\pm=
			D\mp\frac{n{\bf i}}{2}.
		\]
	As a consequence, 
	we 
		obtain the integral version of the fundamental Lichnerowicz formula: 
		\begin{equation}\label{partshyp}
			\int_\Omega\left(|\nabla^{\pm} \Psi|^2-| D^{\pm}\Psi|^2+\frac{R_g+n(n-1)}{4}|\Psi|^2\right)dM={\rm Re}\int_{\partial\Omega}
			\left\langle {{\mathcal W}}^{\pm}(\nu)\Psi,\Psi\right\rangle d\Sigma.
		\end{equation}
		Here, $\Psi\in\Gamma(\mathbb SM)$, $\Omega\subset M$ is a compact domain with a nonempty boundary $\partial\Omega$, which we assume endowed with its inward pointing unit normal $\nu$, and 
			\[
			{\mathcal W}^{\pm}(\nu)=-(\nabla_\nu^{\pm}+\c(\nu) D^{\pm}).
		\]

		A key step in our argument is to rewrite the right-hand side of (\ref{partshyp}) along the portion of $\partial \Omega$ lying on $\Sigma=\partial M$ in terms of the corresponding extrinsic geometry. To proceed, note that $\mathbb SM|_\Sigma$ becomes a Dirac bundle if endowed with the Clifford multiplication
		\begin{equation}\label{cliff:mult:ext}
		\c^{\intercal}(X)\Psi=\c(X)\c(\nu) \Psi,
		\end{equation}
		and the connection 
		\begin{equation}\label{cliff:conn:ext}
			\nabla^{\intercal}_X\Psi  =  \nabla_X\Psi+\frac{1}{2}\c^{\intercal}(\nabla_X\nu)\Psi,
		\end{equation}
		so
		the corresponding Dirac operator $D^{\intercal}:\Gamma(\mathbb \mathbb SM|_\Sigma)\to\Gamma(\mathbb SM|_\Sigma)$ is
		$$
		D^{\intercal}=\c^{\intercal}\circ\nabla^{\intercal}.
		$$
	It follows that 
		\[
		D^{\intercal}-\frac{H_g}{2}=-(\nabla_\nu+\c(\nu)D),
		\]
		which combined with (\ref{partshyp}) yields the following important result for our arguments.
		
		\begin{proposition}\label{intpartf} Under the conditions above,
			\begin{eqnarray}\label{parts3}
				 \int_\Omega\left(|\nabla^{\pm} \Psi|^2-|D^{\pm}\Psi|^2+\frac{R_g+n(n-1)}{4}|\Psi|^2\right)d\Omega
				& = &\int_{\partial\Omega\cap\Sigma}
				\left(\langle D^{\intercal,\pm}\Psi,\Psi\rangle-\frac{H_g}{2}|\Psi|^2\right) d\Sigma\nonumber\\
				& & \quad + {\rm Re}\int_{\partial\Omega\cap {\rm int}\,M}
				\left\langle {{\mathcal W}}^{\pm}(\nu)\Psi,\Psi\right\rangle d\partial \Omega,
			\end{eqnarray}
			where
			\begin{equation}\label{newdirac}
				D^{\intercal,\pm}=D^{\intercal}\pm \frac{(n-1)\bf i}{2}\c(\nu):\Gamma(\mathbb S{\Sigma})\to
				\Gamma(\mathbb S{\Sigma}).
			\end{equation}
		\end{proposition}

		\subsection{$\theta$-boundary conditions}\label{chirbd}
		To further simplify the integral over $\partial\Omega\cap\Sigma$ in (\ref{parts3}) we must impose suitable boundary conditions on $\Psi$.

		\begin{definition}\label{chiralop}
			A {\em chirality operator}	on a spin manifold $(M,g)$ is a (pointwise) self-adjoint involution $Q:\Gamma(\mathbb SM)\to\Gamma(\mathbb SM)$ which is parallel and anti-commutes with Clifford multiplication by tangent vectors.
		\end{definition}

		To simplify the exposition, we henceforth assume that  $n$ is even, as in this case  it is well-known that  Clifford multiplication by the complex volume element $\omega={\bf i}^{n/2}\c(e_1)\cdots \c(e_n)$
		provides a natural chirality operator, namely, $Q=\omega$. Nevertheless, see Remark \ref{n:odd} below for the indication on how a chirality operator may be constructed in odd dimensions as well, so that the argument below may be carried out in full generality.

			We now fix $\kappa\in(0,1]$ and set $\tau=\pm\sqrt{1-\kappa^2}\in(-1,1)$, so that $e^{{\bf i}\theta}=\kappa+\tau{\bf i}$, where $\kappa=\cos\theta$ and $\tau=\sin\theta$, $\theta\in(-\pi/2,\pi/2)$.  
			
		\begin{definition}\label{theta:cond}
		The $\theta$-{\em boundary  operator} $Q_{\theta,g}:\Gamma(\mathbb SM|_{\Sigma})\to\Gamma(\mathbb SM|_{\Sigma})$
		associated to $Q=\omega$ as above is 
		 	\begin{equation}\label{mu:bd:2}
		 	Q_{\theta,g}=e^{{\bf i}\theta Q}Q\c(\nu). 
		 \end{equation} 
		\end{definition}
	
		\begin{remark}\label{theta:v}
			The involutiveness of $Q$ ($Q^2=I$) implies that  $\sin (\theta Q)=(\sin\theta)Q$ and $\cos(\theta Q)=(\cos \theta) I$, so that  
				\begin{equation}\label{mu:bd}
				Q_{\theta,g}=\kappa Q\c(\nu)+\tau{\bf i}\c(\nu).
				\end{equation} 
			Hence, as $\theta$ varies between $\theta=0$ and $\theta=\pm\pi/2$, $Q_{\theta,g}$ interpolates between 
				the chirality boundary operator $Q\c(\nu)$ used in \cite{almaraz2020mass}
			and the MIT bag boundary operators $\pm{\bf i}\c(\nu)$ used in \cite{almaraz2014positive}.
		. 
			\end{remark}

			\begin{remark}\label{geom:mean:p}
			The complex  phase $e^{{\bf i}\theta}=\kappa+\tau{\bf i}$ acquires a geometric meaning if we impose that
			\begin{equation}\label{geo:mean}
				\lambda_s=\sin \theta,
				\end{equation}
			where $s\in\mathbb R$ is the parameter appearing in Proposition \ref{secformSigma}. This matching of parameters, which we always assume to hold throughout this work,  means that $\Sigma_s$ has constant mean curvature equal to $(n-1)\tau$ and is intrinsically isometric to $\mathbb H^{n-1}(-\kappa^2)$, the hyperbolic $(n-1)$-space with curvature $-\kappa^2$. From this perspective, the fundamental trigonometric identity 
			\begin{equation}\label{trig:fund}
					\kappa^2+\tau^2=1
				\end{equation}
		 is just Gauss equation in disguise.
			\end{remark}
		
		\begin{remark}\label{gen:desc:theta}
				The extrinsic data $(\mathbb SM|_{\Sigma},\c^{\intercal},\nabla^{\intercal})$ associated to a hypersurface embedding $\Sigma\hookrightarrow M$ can be identified to objects constructed out of the {\em intrinsic}  data $(\mathbb S\Sigma,\c^\gamma,\nabla^\gamma)$, where $\gamma=g|_\Sigma$ is the induced metric along $\Sigma$; see \cite[Section 2.2]{morel2001eigenvalue} or \cite[Section 2.2]{hijazi2015holographic}. Since we are assuming that $n$ is even, 
				the fact that $Q=\omega$ as above is an involution may be used to split the spinor bundle of $M$ as an orthogonal direct sum of its $\pm 1$-eigenbundles. In this way we obtain the  {\em chiral decomposition}
				\[
					\mathbb SM=\mathbb SM^+\oplus\mathbb SM^-,
				\]
				so that  
				\[
					Q\Psi^\pm=\pm\Psi^\pm,\quad \Psi=\Psi^++\Psi^-,\quad \Psi^\pm\in\Gamma(\mathbb S M^\pm).
				\]
				Upon restriction to $\Sigma$ this yields the first identification, namely, 
			\begin{equation}\label{split:bd}
				\mathbb SM|_{\Sigma}=\mathbb SM^+|_{\Sigma}\oplus \mathbb SM^-|_{\Sigma}=\mathbb S\Sigma\oplus\mathbb S\Sigma,
			\end{equation}
			with
			\begin{equation}\label{ident}
				\Psi\in\Gamma(	\mathbb SM|_{\Sigma})\mapsto 
				\left(
				\begin{array}{c}
					\Psi_1\\
					\Psi_2
				\end{array}
				\right),
				\quad \Psi_i\in \Gamma(\mathbb S\Sigma),
			\end{equation}
			satisfying
			\begin{equation}\label{q:action}
				Q\left(
				\begin{array}{c}
					\Psi_1\\
					\Psi_2
				\end{array}
				\right)=
				\left(
				\begin{array}{c}
					\Psi_1\\
					-	\Psi_2
				\end{array}
				\right).
			\end{equation}
			Moreover, under (\ref{ident}), 
			\[
			\c^\intercal=
			\left(
			\begin{array}{cc}
				\c^\gamma & 0\\
				0 & - \c^\gamma
			\end{array}
			\right),
			\quad 
				\nabla^\intercal=
			\left(
			\begin{array}{cc}
				\nabla^\gamma & 0\\
				0 &  \nabla^\gamma
			\end{array}
			\right),
			\]
so		that 
			\[
				D^\intercal=
			\left(
			\begin{array}{cc}
				D^\gamma & 0\\
				0 & - D^\gamma
			\end{array}
			\right),
			\]
		where $D^\gamma=\c^\gamma\circ\nabla^\gamma$ is the intrinsic Dirac operator associated to $\gamma$. 
			Finally, we may  agree that 
			\begin{equation}\label{cliff:prod}
				\c(\nu)=-{\bf i}\left(
				\begin{array}{cc}
					0 & I\\
					I & 0
				\end{array}
				\right).
			\end{equation}
		\end{remark}
		
			The identifications in Remark \ref{gen:desc:theta} provide a rather  explicit description of the $\theta$-boundary operator which will be useful later on.
		
		\begin{proposition}\label{comp:0}
			With the notation in Remark \ref{gen:desc:theta},  the action of $Q_{\theta,g}$ on a (restricted) spinor is  
				\begin{equation}\label{comp:1}
				Q_{\theta,g}
				\left(
				\begin{array}{c}
					\Psi_1\\
					\Psi_2
				\end{array}	
				\right)=
				\left(
				\begin{array}{c}
					-{\bf i}e^{{\bf i}\theta}\Psi_2\\
					{\bf i}e^{-{\bf i}\theta}\Psi_1
				\end{array}	
				\right)=
				\left(
				\begin{array}{c}
					(\tau-\kappa{\bf i})\Psi_2\\
					(\tau+\kappa{\bf i})\Psi_1
				\end{array}	
				\right).
			\end{equation}
		\end{proposition}
	\begin{proof}
			Use (\ref{mu:bd}), (\ref{q:action}) and (\ref{cliff:prod}). 
		\end{proof}

		\begin{remark}\label{n:odd}
			In the odd dimensional case, we may always define a chirality operator in the direct sum bundle $\mathbb SM\oplus\mathbb SM$ by simply switching the factors. Even though we do not carry out the details, it turns out that this simple trick allows us to straightforwardly extend our main results here to this case; see \cite{almaraz2020mass} for the pertinent details when $\theta=0$.
			\end{remark}

		For later reference, we now isolate a few algebraic facts concerning this formalism. In what follows, we denote by $[A,B]$ (respectively, $\{A,B\}$) the commutator (respectively, the anti-commutator) of the operators $A$ and $B$. 
		
		\begin{proposition}\label{alg:form}
			The following properties hold:
			\begin{enumerate}
				\item $Q_{\theta,g}$ is a self-adjoint involution;
				\item $\{Q,Q_{\theta,g}\}=0$;
				\item $\{D^\intercal,Q_{\theta,g}\}=0$;
				\item $\{\c(\nu),Q_{\theta,g}\}=-2\tau{\bf i}$;
				\end{enumerate}
			\end{proposition}
		
		\begin{proof}
			This already follows (for $n$ even) from the explicit description of these objects in Remark \ref{gen:desc:theta} and Proposition \ref{comp:0}, but we provide here an alternate abstract reasoning that works whenever a chirality operator is available. First, self-adjointness of $Q_{\theta,g}$ follows from it being a {\em real} linear combination of  $Q\c(\nu)$ and ${\bf i}\c(\nu)$, which are easily checked to meet this property.
			Since $\{Q,\c(\nu)\}=0$ we have $\c(\nu) e^{i\theta Q}=e^{-i\theta
				Q}\c(\nu)$ and by (\ref{mu:bd:2}),
			\begin{eqnarray*}
				Q_{\theta,g}^2 & = & e^{{\bf i}\theta Q}Q\c(\nu) e^{{\bf i}\theta Q}Q\c(\nu)\\
				& = & Q e^{{\bf i}\theta Q}\c(\nu) e^{{\bf i}\theta Q}Q\c(\nu)\\
				& = & Q \c(\nu) e^{-{\bf i}\theta Q} e^{{\bf i}\theta Q}Q\c(\nu) \\
				& = & Q\c(\nu) Q\c(\nu)\\
				& = & -Q^2\c(\nu)^2,
			\end{eqnarray*}
		from which the involutiveness of  $Q_{\theta,g}$ follows. The second item is an immediate consequence of the definitions.
			To proceed, note that $[D^\intercal,Q]=0$, so that $[D^\intercal,e^{i\theta Q}]=0$ as well. Also, $\{D^\intercal,\c(\nu) \}=0$, so that
		\begin{eqnarray*}
			D^\intercal Q_{\theta,g} & = & D^\intercal e^{{\bf i}\theta Q}Q\c(\nu)\\
			&  = & e^{{\bf i}\theta Q}QD^\intercal\c(\nu)\\
			& = & -e^{{\bf i}\theta Q}Q\c(\nu) D^\intercal,
		\end{eqnarray*}
		which gives the third item. Finally, the fourth item is proved by first checking that $\{\c(\nu),Q\c(\nu)\}=0$ and $\{\c(\nu),{\bf i}\c(\nu)\}=-2{\bf i}$ and then using this in (\ref{mu:bd}).
		\end{proof}
		
		Since $Q_{\theta,g}$ is a self-adjoint involution, we may consider the projections
		\[
		P^{(\pm)}_{\theta}=\frac{1}{2}\left({\rm I}_{\mathbb SM|_{\Sigma}}\pm Q_{\theta,g}\right):\Gamma(\mathbb SM|_\Sigma)\to\Gamma(V_\theta^{(\pm)})
		\] 
		onto the $\pm 1$-eigenbundles $V_\theta^{(\pm)}$ of $Q_{\theta,g}$. Thus, $\Psi\in\Gamma(V_\theta^{(\pm)})$ if and only if $Q_{\theta,g}\Psi=\pm\Psi$. 
For any $\Psi\in\Gamma(\mathbb SM|_\Sigma)$ we set $\Psi_\theta^{(\pm)}=P^{(\pm)}_{\theta}\Psi\in\Gamma(V_\theta^{(\pm)})$, so that 
		\[
		\Psi=\Psi_\theta^{(+)}+\Psi_\theta^{(-)},
		\]
		an orthogonal decomposition. From Proposition \ref{alg:form}, (3), we have 
		$D^\intercal P_\theta^{(\pm)}=P_\theta^{(\mp)}D^\intercal$, which gives 
		\begin{equation}\label{idpm}
			\langle D^{\intercal} \Psi,\Psi\rangle=\langle D^{\intercal} \Psi_\theta^{(+)},\Psi_\theta^{(-)}\rangle+\langle D^{\intercal} \Psi_\theta^{(-)},\Psi_\theta^{(+)}\rangle.
		\end{equation}
		
		\begin{definition}\label{chirality}
		We say that $\Psi\in\Gamma(\mathbb SM)$ satisfies a $\theta$-{\em boundary condition} if any of the identities 
		\begin{equation}\label{theta:bd:pm}
		Q_{\theta,g}\Psi=\pm\Psi
		\end{equation}
		holds everywhere along $\Sigma=\partial M$ (equivalently, 	$\Psi_\theta^{(\mp)}=0$ ).
		\end{definition}
		
		\begin{proposition}\label{streg}
			If $\Psi$ as in Proposition \ref{intpartf} satisfies the $\theta$-boundary condition  (\ref{theta:bd:pm}) then
			\begin{eqnarray}\label{parts4}
				{\rm Re}\int_{\partial\Omega\cap {\rm int}\,M}
				\left\langle {{\mathcal W}}^{\pm}(\nu)\Psi,\Psi\right\rangle d\partial\Omega
				& = & 
				\int_\Omega\left(|\nabla^{\pm} \Psi|^2-| D^{\pm}\Psi|^2+\frac{R_g+n(n-1)}{4}|\Psi|^2\right)d\Omega\nonumber\\
				&  &\quad +\int_{\partial\Omega\cap\Sigma}
				\frac{1}{2}\left(H_g-(n-1)\tau\right)|\Psi|^2 d\Sigma.
			\end{eqnarray} 
		\end{proposition}
		
		\begin{proof}
		From Proposition \ref{alg:form}, (4), we find that 
			\begin{eqnarray*}
			{\bf i}\langle \c(\nu) \Psi,\Psi\rangle & = & \pm{\bf i}\langle c(\nu)
			Q_{\theta,g}\Psi,\Psi\rangle\\
			& = & \mp {\bf i}\langle Q_{\theta,g} c(\nu) \Psi,\Psi\rangle\pm 2\tau
			|\Psi|^2\\
			& = & \mp{\bf i}\langle \c(\nu)\Psi,Q_{\theta,g}\Psi\rangle\pm 2\tau
			|\Psi|^2\\
			& = & \mp{\bf i}\langle \c(\nu) \Psi,\pm\Psi\rangle\pm 2\tau |\Psi|^2\\
			&= & -{\bf i}\langle \c(\nu) \Psi,\Psi\rangle\pm2\tau |\Psi|^2,
		\end{eqnarray*}
	which gives 
	\begin{equation}\label{ref:need}
		{\bf i}	\langle \c(\nu)\Psi,\Psi\rangle=\pm\tau|\Psi|^2.
		\end{equation}
			On the other hand, from (\ref{idpm}) we have  $\langle D^{\intercal}\Psi,\Psi\rangle=0$. Thus, from (\ref{newdirac}),  
			\begin{eqnarray*}
			\langle D^{\intercal,\pm}\Psi,\Psi\rangle 
			& = & \pm\frac{n-1}{2}{\bf i}\langle\c(\nu)\Psi,\Psi\rangle\\
			& = & \left(\pm\frac{n-1}{2}\right)\left(\pm\tau|\Psi|^2\right),
			\end{eqnarray*}
		so that 
		\begin{equation}\label{match:sign}
			\langle D^{\intercal,\pm}\Psi,\Psi\rangle =  \frac{(n-1)\tau}{2}|\Psi|^2.
		\end{equation}
			 Together with (\ref{parts3}), this  proves (\ref{parts4}).
		\end{proof}
		
	\begin{remark}\label{match:sign:2}
		The sign cancellation leading to (\ref{match:sign}) shows that the integral over $\partial \Omega\cap\Sigma$ in (\ref{parts4}) only has the expected shape if the sign of the Killing connection defined in (\ref{kill:conn})	matches the sign of the $\theta$-boundary condition in (\ref{theta:bd:pm}). In the following we always assume that this sign  convention pairing objects constructed out of the Killing connection $\nabla^\pm$ to the $\theta$-boundary condition involving the projection  $P_\theta^{(\pm)}$  holds true both in statements and computations; see Remark \ref{match:ex} for an illustrative example.  
		\end{remark}

	We now observe that, by Proposition \ref{alg:form} (2), the eigenbundles $V^{(\pm)}_\theta$ are interchanged by $Q$, so that
		\[
		{\rm rank}\,V^{(\pm)}_\theta=\frac{1}{2}{\rm rank}\,\mathbb SM|_\Sigma.
		\]
			It then follows from the analysis in \cite[Subsection 1.4.5]{gilkey2003asymptotic} that
		the projections $P_{\theta}^{(\pm)}$ define {\em elliptic} boundary conditions for the Dirac operator $D^{\pm}$ considered above.
This is the key input in establishing the following existence result.

		\begin{proposition}\label{existspin}
			Let $(M,g,\Sigma)$ be $s$-AH as in Definition \ref{def:as:hyp} with $s$ satisfying (\ref{geo:mean}) and assume further that $R_g\geq -n(n-1)$ and $H_g\geq (n-1)\lambda_s$. Then 
			for any $\Phi\in \Gamma(\mathbb SM)$ such that $D^{\pm}\Phi\in L^2(\mathbb SM)$ there exists a unique $\Xi\in L_1^2(\mathbb SM)$ solving the boundary value problem
			\begin{equation}\label{exist:pr}
			\left\{
			\begin{array}{lcc}
				{D}^{\pm}\Xi=-D^{\pm}\Phi & {\rm in} & M, \\
				Q_{\theta,g}\Xi=\pm \Xi & {\rm on}  & \Sigma.
			\end{array} 
			\right.
			\end{equation}
		\end{proposition}
		
		\begin{proof}
			As already remarked, (\ref{geo:mean}) means that $(M,g,\Sigma)$ is modeled at infinity on $(\mathbb H^n_s,b,\Sigma_s)$. With this information at hand, the proof is a simple adaptation of the argument leading to \cite[Proposition 4.7]{almaraz2020mass}, which treats the case $s=0$.
		\end{proof}
	
		\begin{remark}\label{match:ex}
		The statement of Proposition \ref{existspin} displays an example of the sign matching convention alluded to in Remark \ref{match:sign:2}. For instance, the plus sign in $D^+=\c\circ\nabla^+$ should match the plus sign in the right-hand side of the bottom line of (\ref{exist:pr}) and similarly for $D^-$ and the minus sign. By convention, the remaining possibilities do not occur in (\ref{exist:pr}). This  illustrates the kind of sign convention we henceforth adopt.   
	\end{remark}
	
		\subsection{Imaginary Killing spinors}\label{killing}
		
		We start by recalling a well-known definition which in a sense justifies the consideration of the Killing connection 	$\nabla^{\pm}$ in (\ref{kill:conn}).
		
		\begin{definition}\label{defkill}
			We say that $\Phi\in\Gamma(\mathbb SM)$ is an {\em imaginary Killing spinor} if it is parallel with respect to 
			$\nabla^{\pm}$, that is,
			\[
			\nabla_X\Phi\pm\frac{\bf i}{2}\c(X)\Phi=0,\quad X\in\Gamma(TM).
			\]
			The space of all such spinors is denoted by $\mathcal K^{g,\pm}(\mathbb SM)$.
		\end{definition}

		\begin{remark}\label{type:I:gen}
			 If $\{\mathfrak e_i\}_{i=1}^n$ is a local orthonormal basis of tangent vectors, it is known that for any $\Phi\in \mathcal K^{g,\pm}(\mathbb SM)$ the quantity
			\[
			q_\Phi:=|\Phi|^4+\sum_{i=1}^n\langle\c(\mathfrak e_i)\Phi,\Phi\rangle^2
			\]
			is a non-negative constant \cite[Lemma 5]{baum1989complete}. If  $q_\Phi=0$ we say that $\Phi$ is of {\em type I}.  
			\end{remark}
		
		\begin{remark}\label{killhyp}
			The hyperbolic $n$-space $\mathbb H^n$ can be described in terms of the so-called {\em Poincar\'e ball model}, which is given by the unit $n$-disk 
				\[
				\mathbb B^n=\left\{y\in\mathbb R^n;|y|_\delta<1\right\}
				\]
				endowed with the conformal metric
				\[
				\widehat b=\Omega(y)^{-2}\delta, \quad \Omega(y)=\frac{1-|y|_\delta^2}{2}. 
				\]
				It is easy to check that under the corresponding isometry $\mathcal I:\mathbb H^n(-1)\to \mathbb B^n(1)$, the equidistant hypersurface $\Sigma_s$ is mapped onto $\widehat\Sigma_s=V_{(1)}^{-1}(s)$, where here $V_{(1)}=\Omega(y)^{-1}y_1$. 
				More importantly for our purposes,
				this conformal relation between $(\mathbb B^n,\widehat b)$ and $(\mathbb B^n,\delta)$  allows us to canonically identify the corresponding spin bundles $\mathbb S\mathbb B^n_\delta$ and $\mathbb S\mathbb B^n_{\widehat b}$, so that $\phi\in\Gamma(\mathbb S\mathbb B^n_\delta)$ corresponds to a certain $\overline{\phi}\in\Gamma(\mathbb S\mathbb B^n_{\widehat b})$. Under this identification, if $u\in\Gamma(\mathbb S\mathbb B^n_\delta)$ is a $\nabla^\delta$-{\em parallel} spinor then
				the prescription  
				\begin{equation}\label{presc}
					\Phi_{u,\pm}(y):=\Omega(y)^{-1/2}\overline{\left(I\mp {\bf{i}} \c^{\delta}(y)\right)u}\in\Gamma(\mathbb S\mathbb B^n_{\widehat b})
				\end{equation}
				exhausts the space $\mathcal K^{\widehat b,\pm}(\mathbb S\mathbb B^n)$ \cite[Theorem 1]{baum1989complete}. As usual, we assume that $\Psi_{u,\pm}$ is of {type I}, which here  means that 
				\begin{equation}\label{type:I}
					|u|^4_\delta+\sum_{i=1}^n\langle\c^\delta(\partial_{y_i})u,u\rangle^2
					=0;
				\end{equation} 
				see  Remark \ref{type:I:gen}. 
		\end{remark}

		In general, if $(M,g,\Sigma)$ is $s$-AH with $s$ as in (\ref{geo:mean}), we may consider the space  
		\[
		\mathcal K^{g,\pm,(\pm)_\theta}(\mathbb SM)=\{\Phi\in\mathcal K^{g,\pm}(\mathbb SM);Q_{\theta,g}\Phi=\pm\Phi\}
		\]
		of all imaginary Killing spinors satisfying the corresponding $\theta$-boundary condition along $\Sigma_s$. Notice that again we use here the sign matching convention in Remark \ref{match:sign:2}, so that only two possibilities for the various signs involved actually occur, namely, $\mathcal K^{g,+,(+)_\theta}(\mathbb SM)$ and $\mathcal K^{g,-,(-)_\theta}(\mathbb SM)$. 
		Our next goal is to check that our model space $(M,g,\Sigma)=(\mathbb H^n_s,b,\Sigma_s)$ carries the maximal number of linearly independent such spinors.
	
		\begin{proposition}\label{dimkill}
			We have
			\[
			\dim_{\mathbb C}\mathcal K^{b,\pm,(\pm)_\theta}(\mathbb S\mathbb H^n_s)=2^{k-1}, \quad n=2k.
			\]
In particular, $\mathcal K^{b,\pm,(\pm)_\theta}(\mathbb S\mathbb H^n_s)\neq\{0\}$.
		\end{proposition}
		
\begin{proof}
We use the notation of Remark \ref{killhyp} and represent by $(\mathbb B^n_s,\widehat b,\widehat \Sigma_s)$ the realization of $(\mathbb H^n,b,\Sigma_s)$ in  the Poincar\'e ball model.
Let $\mathbb S\mathbb B^n_{s,\delta}$ be the spinor bundle of $\mathbb B^n_s$ with respect to the flat metric $\delta$ and $\mathbb S\mathbb B^{n\pm}_{s,\delta}$ its chiral factors. 
The identification in (\ref{split:bd}), with $M=\mathbb S\mathbb B^n_{s,\delta}$ and $\Sigma=\widehat\Sigma_s$,  yields a natural bundle isomorphism $\mathbb S\mathbb B^{n+}_{s,\delta}|_{\widehat\Sigma_s}\equiv\mathbb S\mathbb B^{n-}_{s,\delta}|_{\widehat\Sigma_s}$. Since these bundles can be trivialized by $\nabla^\delta$-parallel spinors, this gives 
\[
\ker\nabla^\delta\cap\Gamma(\mathbb S\mathbb B^{n+}_{s,\delta}|_{\widehat\Sigma_s})\equiv
\ker\nabla^\delta\cap\Gamma(\mathbb S\mathbb B^{n-}_{s,\delta}|_{\widehat\Sigma_s}),
\]
so that a further appeal to parallel transport provides the natural identification
\begin{equation}\label{ident:+-}
\ker\nabla^\delta\cap \Gamma(\mathbb S\mathbb B^{n+}_{s,\delta})\equiv \ker\nabla^\delta\cap \Gamma(\mathbb S\mathbb B^{n-}_{s,\delta}).
\end{equation}
On the other hand, we may  define an involution $\mathcal Q_{\theta,\delta}:\ker\nabla^\delta\to \ker\nabla^\delta$  by first  applying $Q_{\theta,\delta}$ to the restriction of  $u=(u^+,u^-)\in \ker\nabla^\delta$ to $\widehat\Sigma_s$ and then extending the resulting spinor back to $\mathbb B^n_s$ by parallel transport. From (\ref{comp:1}) we see that   
		\[
	\mathcal Q_{\theta,\delta}	
		\left(
		\begin{array}{c}
			u^+\\
			u^-
		\end{array}	
		\right)=
		\left(
		\begin{array}{c}
			-{\bf i}e^{{\bf i}\theta}u^-\\
			{\bf i}e^{-{\bf i}\theta}u^+
		\end{array}	
		\right).
		\]
	It is immediate to check that this involution satisfies the following properties:
	\begin{itemize}
		\item $[	\mathcal Q_{\theta,\delta},\c^\delta(y)]=0$ for any $y\in \mathbb B^n_s$; 
		\item
		when restricted to $\widehat\Sigma_s$,  $\mathcal Q_{\theta,\delta}$ corresponds to $Q_{\theta,\widehat b}$ under the identification in Remark \ref{killhyp} in the sense that $Q_{\theta,\widehat b}(\overline u)=\overline{	\mathcal Q_{\theta,\delta}(u)}$.
		\end{itemize}
	 Thus, if $u\in \ker\nabla^\delta$ satisfies $	\mathcal Q_{\theta,\delta}	u=\pm u$, which is allowed by (\ref{ident:+-}), then $\Phi_{u,\pm}$ given by (\ref{presc}) lies in $\mathcal K^{b,\pm,(\pm)_\theta}(\mathbb S\mathbb H_s^n)$ due to the properties above. Since the homomorphism $u\mapsto \Phi_{u,\pm}$ is obviously injective, this completes the proof.
	\end{proof}

			\section{The proof of Theorem \ref{main:pmt}}\label{proof:main}
		
		We make use of Witten's spinorial approach as adapted to the asymptotically hyperbolic spin case \cite{min1989scalar,andersson1998scalar,wang2001mass,chrusciel2003mass,almaraz2020mass}. The first step is the following result, which links imaginary Killing spinors to static potentials in the model space. Recall that we are always assuming that the relation (\ref{geo:mean}) holds true. Also, we set 
		\[
		\mathcal C_{b,s}^\uparrow=\{V\in\mathcal N_{b,s};\langle\langle V,V\rangle\rangle_s=0, \langle\langle V,V_{(0)}\rangle\rangle_s>0 \}
		\]
		to be the future-pointing isotropic cone. 
		
		\begin{proposition}\label{link:stop}
			For any $\Phi\in\mathcal K^{b,\pm,(\pm)_\theta}(\mathbb H^n_s)$ we have that $V_\Phi:=|\Phi|^2\in \mathcal N_{b,s}$. If $\Phi\not\equiv 0$ is of type I then 
			$V_\Phi\in \mathcal C_{b,s}^\uparrow$
			 and moreover any $V\in \mathcal C_{b,s}^\uparrow$ can be written as $V=V_\Phi$ for some such $\Phi$.
			\end{proposition}
		
		\begin{proof}
			We only check the first assertion, since the rest follows as in the proof of \cite[Proposition 5.1]{almaraz2020mass}. 
			Thus, if we start with
			 $\Phi=\Phi_{u,\pm}\in\mathcal K^{b,\pm,(\pm)_\theta}(\mathbb H^n_s)$, 
			 where $\nabla^\delta u=0$ as in 
			 (\ref{presc}), 
			 a computation shows that 
			\begin{equation}\label{exp:ext:im:s}
		V_{\Phi}(y)=|\Phi_{u,\pm}|^2=|u|_\delta^2V_{(0)}(y)\mp{\bf i}\sum_{j=1}^n\langle \c^\delta(\partial_{y_j})u,u\rangle_\delta V_{(j)}(y), 
			\end{equation}
			with $u=u^++u^-$, $u^-=\pm{\bf i}e^{-{\bf i}\theta}u^+$. Here, 
			\begin{equation}\label{exp:stat:ball}
			V_{(0)}(y)=\frac{1+|y|^2_\delta}{1-|y|^2_\delta}, \quad 	V_{(j)}(y)=\frac{2y_j}{1-|y|^2_\delta},
			\end{equation}
			are the expressions of the static potentials in the Poincar\'e ball model. 
			It follows that 
			\begin{eqnarray*}
				\langle \c^\delta(\partial_{y_1})u,u\rangle_\delta 
				& = & \langle \c^\delta(\partial_{y_1})u^+,u^+\rangle_\delta + \langle \c^\delta(\partial_{y_1})u^-,u^-\rangle_\delta\\
				& = & 2\langle \c^\delta(\partial_{y_1})u^+,u^+\rangle_\delta\\
				& = & \langle \c^\delta(\partial_{y_1})v,v\rangle_\delta, 
			\end{eqnarray*}
		where $v=u^+\pm{\bf i}u^+$, so that $\Phi_{v,\pm}\in \mathcal K^{b,\pm,(\pm)_0}(\mathbb H^n_0)$. Noticing that the unit normal to $\Sigma_0$ is everywhere aligned to $\partial_{y_1}$, we may use (\ref{ref:need}) with $\theta=0$ to see  that, along $\Sigma_0$, 
		\[
	\langle	\c^b(\partial_{y_1})\Phi_{v,\pm} ,\Phi_{v,\pm}\rangle_b=0. 
		\]
		Using (\ref{presc}) and taking the real part we get, for $y\in\Sigma_0$,
		\begin{eqnarray*}
			\langle \c^\delta(\partial_{y_1}) v,v\rangle_\delta
			& = & -\langle \c^\delta(\partial_{y_1})\c^\delta(y)v,\c^\delta(y)v\rangle_\delta\\
			& = & \langle \c^\delta(y)\c^\delta(\partial_{y_1})v,\c^\delta(y)v\rangle_\delta\\
			& = & |y|^2_\delta \langle \c^\delta(\partial_{y_1})v,v\rangle_\delta,
		\end{eqnarray*}
		where we used that $\langle \partial_{y_1},y\rangle_\delta=0$ in the second step. This shows that
		$\langle c^\delta(\partial_{y_1})v,v\rangle_\delta=0$, so $V_\Phi$ is a linear combination of $V_{(0)},V_{(2)},\cdots,V_{(n)}$, as desired.
		\end{proof}

	We now take an imaginary Killing spinor $\Phi\in \mathcal K^{b,\pm,(\pm)_\theta}(\mathbb H^n_s)$ so that $V_\Phi\in\mathcal C_{b,s}^\uparrow$ as in Proposition \ref{link:stop}. Using the given chart at infinity $F$ we may   transplant this spinor to the whole of $M$ in the usual way so that the corresponding $\theta$-boundary condition is satisfied along $\Sigma$. It is easy to check that this transplanted spinor, say $\Phi_*$, satisfies $D^\pm\Phi_*\in L^2(\mathbb SM)$, 
so we may apply Proposition \ref{existspin} to obtain $\Xi\in L^2_1(\mathbb SM)$ such that $D^{\pm}\Xi=-D^{\pm}\Phi_*$ and $Q_\theta\Xi=\pm\Xi$ along $\Sigma$. Thus, $\Psi_{\Phi}:=\Phi_*+\Xi$ is Killing harmonic ($D^{\pm}\Psi_{\Phi}=0$), satisfies $Q_\theta\Psi_{\Phi}= \pm\Psi_{\Phi}$ along $\Sigma$ and   asymptotes $\Phi$ at infinity in the sense that $\Psi_\Phi-\Phi\in L^2_1(\mathbb SM)$. We may now state our first main result, which provides a Witten-type formula for the mass functional restricted to the cone $\mathcal C_{b,s}^\uparrow$; compare with \cite[Theorem 5.2]{almaraz2020mass} which considers the case $s=0$. 
	
	\begin{theorem}\label{maintheo}
		With the notation above,
		\begin{eqnarray}\label{maintheo2}
			\frac{1}{4}\mathfrak m_{s,F}(V_\Phi) & = & \int_M\left(|\nabla^{\pm}\Psi_\Phi|^2+\frac{R_g+n(n-1)}{4}|\Psi_\Phi|^2\right)dM\\
			& & \quad +\frac{1}{2}\int_\Sigma (H_g-(n-1)\lambda_s)|\Psi_\Phi|^2d\Sigma,\nonumber
		\end{eqnarray}
		for any $\Phi\in\mathcal K^{b,\pm,(\pm)_\theta}(\mathbb H^n_s)$ so that $V_\Phi\in\mathcal C_{b,s}^\uparrow$.
	\end{theorem}
		
		\begin{proof}
			As in the proof of \cite[Theorem 5.2]{almaraz2020mass}, this follows by applying (\ref{parts4}) to $M_r$, the region in $M$ bounded by $\Sigma_r\cup S_{r,+}^{n-1}$, and checking that the well-known cancellations apply so that  the corresponding left-hand side converges as $r\to+\infty$ to the left-hand-side of (\ref{maintheo2}). The details are omitted. 
		\end{proof}
	
	We now explain how this mass formula implies Theorem \ref{main:pmt}. The dominant energy conditions $R_g\geq -n(n-1)$ and $H_g\geq (n-1)\lambda_s$ combined with  Proposition \ref{link:stop}
	 clearly imply that 
	\begin{equation}\label{eq:case}
	\langle P_s(F),V\rangle_s\geq 0
	\end{equation}
	for any  $V\in\mathcal C_{b,s}^\uparrow$.  Hence, $P_s(F)$ is time-like and future-directed, unless there exists $V\not\equiv 0$ such that the equality holds in (\ref{eq:case}), in which case there exists, by Proposition \ref{link:stop}, a nonzero imaginary Killing spinor on $M$, say $\Psi^\theta$, satisfying the corresponding $\theta$-boundary condition along $\Sigma$. 
	
	\begin{proposition}\label{geom:bd:det}
	The existence of $\Psi^\theta$ implies that $g$ is Einstein (with ${\rm Ric}_g=-(n-1)g$) and $\Sigma$ is totally umbilical (with $H_g=(n-1)\lambda_s$). 	
		\end{proposition}
	
	\begin{proof}
		That $g$ is Einstein is a classical fact \cite[Theorem 8]{baum1991twistors}. In order to check  the assertion regarding $\Sigma$, take $X\in\Gamma(T\Sigma)$, so that $\{\c(X),\c(\nu)\}=0$, which implies that $Q_\theta\c(X)=\c(X)Q_{-\theta}$. In the rest of this argument we set $Q_\theta=Q_{\theta,g}$ for simplicity. It is convenient here to factor out $\c(\nu)$ from $Q_\theta$ so we consider
		\[
		Q^*_{\theta}:=-Q_\theta\c(\nu)=\kappa Q+\tau{\bf i}, 
		\]
		which satisfies $\nabla Q_\theta^*=0$ and $Q^*_\theta\c(Y)=-\c(Y)Q^*_{-\theta}$ for {\em any} $Y\in \Gamma(TM)$. After applying $\nabla^\pm_X$ to the $\theta$-boundary condition $\pm\Psi^\theta=Q_\theta\Psi^\theta$, we find that
		\begin{eqnarray*}
			0 & = & \nabla_X \left(Q^*_{\theta}\c(\nu)\Psi^\theta\right)\pm\frac{\bf i}{2}\c(X)Q^*_\theta\c(\nu)\Psi^\theta\\
			& = & Q_\theta^*\nabla_X(\c(\nu)\Psi^\theta)\mp\frac{\bf i}{2}Q^*_{-\theta}\c(X)\c(\nu)\Psi^\theta\\
			& = &  Q_\theta^*\c(\nabla_X\nu)\Psi^\theta+Q_\theta^*\c(\nu)\nabla_X\Psi^\theta
			\pm\frac{\bf i}{2}Q^*_{-\theta}\c(\nu)\c(X)\Psi^\theta\\
			&  = & Q_\theta^*\c(\nabla_X\nu)\Psi^\theta_0\mp\frac{\bf i}{2}Q_\theta^*\c(\nu)\c(X)\Psi^\theta	\pm\frac{\bf i}{2}Q^*_{-\theta}\c(\nu)\c(X)\Psi^\theta\\
			& = & Q_\theta^*\c(\nabla_X\nu)\Psi^\theta\pm\tau\c(\nu)\c(X)\Psi^\theta.
		\end{eqnarray*}
	We now restore the $\theta$-boundary operator by applying $\c(\nu)$ to this identity. We get 
	\begin{eqnarray*}
	0 & = & -Q_{-\theta}\c(\nabla_X\nu)\Psi^\theta\mp\tau\c(X)\Psi^\theta\\
	& = & -Q_{-\theta}\c(\nabla_X\nu)\Psi^\theta\mp\tau\c(X)Q_\theta\Psi^\theta\\
	& = & -Q_{-\theta}\c(\nabla_X\nu+\tau X)\Psi^\theta,
	\end{eqnarray*}
	which implies $\c(\nabla_X\nu+\tau X)\Psi^\theta=0$  since $Q_{-\theta}$ is invertible.  On the other hand, it is easy to check that 
	\begin{equation}\label{re:im}
	Y|\Psi^\theta|^2=\mp{\bf i}\langle\c(Y)\Psi^\theta,\Psi^\theta\rangle,
	\end{equation} 
	so we obtain 
	\[
	(\nabla_X\nu+\tau X)|\Psi^\theta|^2=\mp{\bf i}\langle\c(\nabla_X\nu+\tau X)\Psi^\theta,\Psi^\theta\rangle=0
	\]
and since $|\Psi^\theta|^2$ is known to vary exponentially along geodesics, we conclude that $-\nabla_X\nu=\tau X$, as desired. 
	\end{proof}
	
This proposition implies that the embedding $\Sigma\hookrightarrow M$ has the same second fundamental form as the model embedding $\Sigma_s\hookrightarrow \mathbb H^n_s$. The next result takes care of the corresponding first fundamental forms. Recall that $\Psi^\theta$ satisfies $\nabla^\pm\Psi^\theta=0$  and $Q_{\theta}\Psi^\theta=\pm \Psi^\theta$. 

\begin{proposition}\label{takes:care}
	Restricted to $\Sigma$, $\Psi^\theta$ satisfies 
	\[
		\nabla^\gamma_X\Psi^\theta\mp\frac{\kappa{\bf i}}{2}\c^\gamma(X)\Psi^\theta=0, \quad X\in\Gamma(T\Sigma).
	\]
	\end{proposition}

\begin{proof}
By (\ref{cliff:conn:ext}), 
\begin{eqnarray*}
	\nabla_X^\intercal
	\Psi^\theta
	& = & \mp\frac{{\bf i}}{2}\c(X)	
	\Psi^\theta
	-\frac{\tau}{2}\c^\intercal(X)
	\\
	& = & \pm\frac{\bf i}{2}\c^\intercal(X)\c(\nu)	
	\Psi^\theta
	-\frac{\tau{\bf}}{2}\c^\intercal(X)	
	\Psi^\theta.
\end{eqnarray*}
	If we  decompose $\Psi^\theta=(\Psi^\theta_1,\Psi^\theta_2)\in\Gamma(\mathbb S M|_\Sigma)$ as in (\ref{ident}) and use 
the identifications in Remark \ref{gen:desc:theta}, we get
\[
\nabla_X^\gamma
\left(
\begin{array}{c}
	\Psi^\theta_1\\
	\Psi^\theta_2
\end{array}
\right)=\pm\frac{1}{2}	\left(
\begin{array}{c}
	\c^\gamma(X)\Psi^\theta_2\\
	-\c^\gamma(X)\Psi^\theta_1
\end{array}
\right)
-\frac{\tau}{2}
\left(
\begin{array}{c}
	\c^\gamma(X)\Psi^\theta_1\\
	-\c^\gamma(X)\Psi^\theta_2
\end{array}
\right)
\]
The result follows if we note that by (\ref{comp:1}) the corresponding $\theta$-boundary condition says that $\Psi^\theta_2=\pm(\tau+\kappa{\bf i})\Psi^\theta_1$.
\end{proof}

Now, from Proposition \ref{geom:bd:det}, Gauss equation and (\ref{trig:fund}) we find that the scalar curvature of $\gamma$ is $R_\gamma=-(n-1)(n-2)\kappa^2$. Also,  
it follows from (\ref{asympthyp}) that $\gamma$ has the appropriate decay at infinity to the hyperbolic metric in 
$\mathbb H^{n-1}(-\kappa^2)$, the hyperbolic $(n-1)$-space with sectional curvature $-\kappa^2$. Hence, $(\Sigma,\gamma)$  is
asymptotically hyperbolic in the sense of \cite{chrusciel2003mass} (that is, as a boundaryless $(n-1)$-manifold having $\mathbb H^{n-1}(-\kappa^2)$ as its model at infinity) and 
therefore has a well defined mass vector. 
Since Proposition \ref{takes:care} and \cite[Theorem 8]{baum1991twistors} imply that $\gamma$ is Einstein with ${\rm Ric}_\gamma=-(n-2)\kappa^2\gamma$, the Ashtekar-Hansen-type formula in \cite[Theorem 3.3]{herzlich2016computing}  may be used to check that this mass vector vanishes. By the rigidity part of the positive mass theorem in \cite{chrusciel2003mass}, $(\Sigma,\gamma)$ is {\em isometric} to $\mathbb H^{n-1}(-\kappa^2)$. Thus, we have seen that the embedding $\Sigma\hookrightarrow M$ has the same first and second fundamental forms as the model embedding $\Sigma_s\hookrightarrow\mathbb H^n_s$. This allows us to glue $(M,g,\Sigma)$ to $(\mathbb H^n_{-s},b,\Sigma_{-s})$ along the common boundary to obtain a (smooth) boundaryless  $n$-manifold which is asymptotically hyperbolic (with $(\mathbb H^n,b)$ as its model at infinity) and Einstein (it actually carries a nonzero imaginary Killing spinor, an appropriate extension of $\Psi^\theta$). Again appealing to \cite{herzlich2016computing,chrusciel2003mass}, we find that this glued manifold is isometric to $(\mathbb H^n,b)$ and hence $(M,g,\Sigma)$ is isometric to $(\mathbb H^n_s,b,\Sigma_s)$, which completes the proof of Theorem \ref{main:pmt}.

\begin{remark}\label{reg:smooth}
	Although the claimed regularity of the glued metric, say $\widehat g$, in the argument above may be verified by means of the standard elliptic machinery as in \cite{deturck1981some}, for the specific gluing above a quite elementary approach is available, as we now describe. In principle, $\widehat g$ is only $C^{1,1}$ along the common boundary $\Sigma$ where the gluing takes place. 
	We note however that 
	as we approach a given point $p\in\Sigma$ from each side then $\widetilde g$ is actually {\em smooth} all the way up to $\Sigma$. Thus, we can choose a small neighborhood $V\subset \Sigma$ of $p$ and neighborhoods $U^{\pm}$ in each side of $\Sigma$ such that $U^+\cap U^-=V$. We set $\widehat g^\pm=\widehat g|_{U^\pm}$ which, being smooth, allows us to consider Fermi coordinates $\{x^\pm_i\}_{i=1}^n$ on $U^\pm$ and `'centered'' at $V$, with the convention that $x^\pm_n$  corresponds to the normal geodesic direction. In the following computations, we assume that, when decorating a geometric invariant, the symbol $\pm$ refers to $\widehat g^\pm$. Also, we use commas to denote partial differentiation with respect to the Fermi coordinates above. By letting greek indices vary from $1$ to $n-1$, we are left with the task of initially checking that $\widehat g^+_{\alpha\beta,nn}=\widehat g^-_{\beta\alpha,nn}$ at $p$, as these are the only second order derivatives of $\widehat g^\pm$ for which this possibly fails to hold (here we use that $\widehat g^\pm_{in}=\delta_{in}$ in these coordinates). We start by recalling that 
	\begin{equation}\label{to:inv}
	{\rm Ric}_{\widehat g^\pm_{ij}}\approx -\frac{1}{2}\widehat g_{\pm}^{kl}\widehat g^{\pm}_{ij,kl}+I^\pm_{ij}, 
	\end{equation} 
	where 
	\begin{equation}\label{chris}
	I^\pm_{ij}=\frac{1}{2}\left(\widehat g^\pm_{ki}\widehat \Gamma^k_{\pm,j}+g^\pm_{kj}\widehat\Gamma^k_{\pm,i}\right),\quad \widehat\Gamma^k_\pm=g^{rs}_\pm\widehat\Gamma_{\pm{rs}}^k,
	\end{equation}
	with $\widehat\Gamma_{\pm{rs}}^k$ being the Christoffel symbols and $\approx$ meaning here that we are discarding terms of at most first order in $\widehat g^\pm$, as they agree to each other as we approach $p$. By inverting (\ref{to:inv}) we obtain 
	\[
	\widehat g^{\pm}_{\alpha\beta,nn}\approx -2\widehat g^\pm_{nn}{\rm Ric}_{\widehat g_{\alpha\beta}^\pm}+2\widehat g^\pm_{nn}I^\pm_{\alpha\beta}\approx 2\widehat g^\pm_{nn}I^\pm_{\alpha\beta},
	\]  
	where in the last step we used that $\widehat g_\pm$ is Einstein with ${\rm Ric}_{\widehat g_{ij}^\pm}=-(n-1)\widehat g^\pm_{ij}$. 
	By (\ref{chris}),  
	\begin{eqnarray*}
		\widehat g^{\pm}_{\alpha\beta,nn} 
		& \approx & \widehat g^\pm_{nn}\left(\widehat g^\pm_{\phi\alpha}\widehat \Gamma^\phi_{\pm,\beta}+g^\pm_{\phi\beta}\widehat\Gamma^\phi_{\pm,\alpha}\right) \\
		& \approx & \widehat g^\pm_{nn}\widehat g_\pm^{nn}\left(\widehat g^\pm_{\phi\alpha}\widehat\Gamma_{\pm{nn,\beta}}^\phi+
		\widehat g^\pm_{\phi\beta}\widehat\Gamma_{\pm{nn,\alpha}}^\phi
		\right)\\
		& & \quad +
		\widehat g^\pm_{nn}\left(\widehat g^\pm_{\phi\alpha}
		\widehat g_\pm^{\varphi\epsilon}
		\widehat\Gamma_{\pm{\varphi\epsilon,\beta}}^\phi+
	\widehat g^\pm_{\phi\beta}
	\widehat g_\pm^{\varphi\epsilon}
	\widehat\Gamma_{\pm{\varphi\epsilon,\alpha}}^\phi
		\right).
	\end{eqnarray*}
Since 
\[
\widehat\Gamma_{\pm{nn}}^\phi=\frac{1}{2}g_\pm^{\phi i}\left(2g^\pm_{in,n}-g^\pm_{nn,i}\right)
\]
vanishes identically,
we see that the first term in the right-hand side above vanishes identically as well. On the other hand, by \cite[Lemma 1.2]{deturck1981some} we may assume that $\{x_\alpha\}$ is harmonic with respect to $\widehat g^\pm|_V$, so that the second term vanishes in the limit as we approach $p$. 
We conclude that $\widehat g^{+}_{\alpha\beta,nn}=	\widehat g^{-}_{\alpha\beta,nn}$ at $p$, that is, $\widehat g_{\alpha\beta}$ is $C^2$ (in fact, smooth), as desired.
   	\end{remark}

\section{The horospherical case}\label{horo:case}

Here we briefly discuss how the methods above can be modified to establish a positive mass theorem whose rigidity statement implies the  horospherical case in \cite[Theorem 2]{souam2021mean}; see Theorem \ref{main:th:horo} below and compare with Corollary \ref{souam:cor} and Remark \ref{horo}.  Roughly speaking, this corresponds to taking the limits $\theta\to\pm\pi/2$ in the results from the previous sections, but new phenomena emerge upon analysis of this formal limit, so a separate proof is required; see Remark \ref{limit:th} for more on this point. Needless to say, at least from a conceptual viewpoint the approach below is quite similar to the one leading to our previous main result (Theorem \ref{main:pmt}) so our presentation will be brief at some points.   		

By means of the hyperboloid model $\mathbb H^n\hookrightarrow\mathbb R^{1,n}$ described in Section \ref{intro}, we consider the {\em horoball}
\[
\mathbb H^n_{h,\chi}=\left\{x\in\mathbb H^n;V_{h}(x)\leq \chi\right\},\quad \chi>0,
\] 
where $V_{h}=V_{(0)}-V_{(1)}$. We denote by $\Sigma_{h,\chi}$ its boundary, so that as $\chi$ varies we obtain the so-called {\em horospherical foliation} of $\mathbb H^n$. This terminology is justified by the following result. 

\begin{proposition}\label{desc:horo}
The triple	$(\mathbb H^n_{h,\chi},b,\Sigma_{h,\chi})$ is a static domain whose boundary $\Sigma_{h,\chi}$ is a horosphere. In this case, $(\widetilde\Lambda,\widetilde\lambda)=(-n(n-1)/2,n-1)$
	and the corresponding space of static potentials $\mathcal N_{b,h}^{\chi}$ has dimension $n$ and is generated by $V_{h}, V_{(2)},\cdots, V_{(n)}$. 
	\end{proposition}

\begin{proof}
	Along $\Sigma_{h,\chi}$ we have $|\nabla^bV_h|=\chi$ so its outward pointing unit normal is $\eta_h=\chi^{-1}\nabla^bV_h$. As in the proof of Proposition \ref{secformSigma} we compute that 
	\begin{equation}\label{horo:ind}
	\nabla^b_X\eta_h=X, \quad X\in\Gamma(T\Sigma_{h,\chi}),
	\end{equation}
	which proves that $\Sigma_{h,\chi}$ is a horosphere indeed (with constant mean curvature equal to $n-1$). In particular, $\mathcal N_{b,h}^{\chi}$ is formed by
	those static potentials of $(\mathbb H^n,b)$ which, when 
	 restricted to $\Sigma_{h,\chi}$,   satisfy $\partial V/\partial\eta_h=V$. A calculation using (\ref{formgrad}) shows that, restricted to $\Sigma_{h,\chi}$,
	\[
	\frac{\partial V_{(0)}}{\partial\eta_h}=x_0-\chi^{-1},\quad \frac{\partial V_{(1)}}{\partial\eta_h}=x_1-\chi^{-1},\quad \frac{\partial V_{(j)}}{\partial\eta_h}=x_j, \quad j\geq 2,	
	\] 
	which completes the proof. 
\end{proof}

In analogy with the discussion preceding Conjecture \ref{conjmp}, 
we now seek to determine the structure of the natural representation of the relevant subgroup of isometries of $(\mathbb H^n,b)$  on the space of static potentials $\mathcal N_{b,h}^{\chi}$. We start with the {\em parabolic} subgroup $\mathcal P$ of  isometries intertwining the horospherical leaves $\Sigma_{h,\chi}$ above.  
In the hyperboloid model $\mathbb H^n\hookrightarrow\mathbb R^{1,n}$, it may be accessed as follows \cite[Section 10.2]{taylor1986noncommutative}. 
As in Remark \ref{iso:desc}, any element of the isometry group $O^\uparrow(1,n)$ of $(\mathbb H^n,b)$, viewed as (the restriction of) a time-orientation preserving linear isometry of $(\mathbb R^{1,n},\langle\,,\rangle_{1,n})$, defines a conformal diffeomorphism of the sphere at infinity ${\mathbb S}^{n-1}_\infty$  of $\mathbb H^n$, and conversely. Here, we view  $\mathbb S^{n-1}_\infty$ as the set of isotropic lines in the future-directed cone $\mathcal C^\uparrow:=\{x\in\mathbb R^{1,n};\langle x,x\rangle_{1,n}=0, x_0>0\}$. Under this identification, the horospherical foliation $\Sigma_{h,\chi}$ picks the distinguished element $o=\{x\in\mathbb R^{1,n};x_0=x_1\}\cap\mathcal C^\uparrow\in\mathbb S^{n-1}_\infty$. In this picture, $\mathcal P$ emerges as the stabilizer of $o$ under this conformal action of $O^\uparrow(1,n)$  on  ${\mathbb S}^{n-1}_\infty$, so that $\mathbb S^{n-1}_\infty=O^\uparrow(1,n)/\mathcal P$, the homogeneous flat model of Conformal Geometry. 
Explicitly, we have the {\em Langlands decomposition}
\[
\mathcal P=O(n-1)\mathbb R N,
\] 
where $O(n-1)$ acts by rotations on $[\partial_{x_2},\cdots,\partial_{x_n}]$,
$\mathbb R$ comes from boosts 
\[
\mathfrak b_{\varrho}=\left(
\begin{array}{cc}
\cosh\varrho & \sinh\varrho\\
\sinh\varrho & \cosh\varrho	
	\end{array}
\right), \quad \varrho\in\mathbb R,
\]
acting on the Lorentzian plane $[\partial_{x_0},\partial_{x_1}]$ and the last factor $N$ comprises the isomorphic image of $\mathbb R^{n-1}$ under exponentiation:
\[
U\in\mathbb R^{n-1}\stackrel{\rm exp}{\longrightarrow}  \left(
\begin{array}{ccc}
	1+\frac{|U|^2}{2} & -\frac{|U|^2}{2} & U\\
	\frac{|U|^2}{2} & 1-\frac{|U|^2}{2} & U\\
	U^t & -U^t & I_{n-1}
\end{array}
\right)\in N.
\]
Here,  
 the superscript means transpose of a row vector.
Since $N=\mathbb R^{n-1}\subset \mathcal P$ is normal and intersects $O(n-1)\mathbb R$ only at the identity element, we actually have the semi-direct product representation
\[
\mathcal P=\left(0(n-1)\mathbb R\right)\ltimes\mathbb R^{n-1}. 
\] 
As can be easily checked, each $\mathfrak b_\varrho$ 
induces an (ambient) isometry between  $(\mathbb H^n_{h,\chi},b)$ and $(\mathbb H^n_{h,e^{-\varrho}\chi},b)$ corresponding to a multiplicative shift of $e^{-\varrho}$ on the parameter $\chi$. We conclude that  $\mathcal P^h:=\mathcal P/\mathbb R=O(n-1)\ltimes\mathbb R^{n-1}$, the group of Euclidean isometries in dimension $n-1$, is the full isometry group of each $(\mathbb H^n_{h,\chi},b)$. In particular, since the $\mathcal P^h$-action on $\Sigma_{h,\chi}=\partial \mathbb H^n_{h,\chi}$ is transitive, this confirms that each  horosphere is 
{\em intrinsically} flat.
As in Remark \ref{iso:desc}, we obtain a natural representation $\rho^{h}$ of $\mathcal P^h$ on $\mathcal N_{b,h}^\chi$ by setting $\rho^h_A(V)=V\circ A^{-1}$.    

\begin{proposition}\label{action:parab} The space of static potentials $\mathcal N_{b,h}^\chi$ splits into two irreducible representations under $\rho^{h}$,   namely,
	\begin{equation}\label{sp:irr}
	\mathcal N_{b,h}^{\chi}=[V_h]\oplus[V_{(2)},\cdots,V_{(n)}],
	\end{equation}
	with $\rho^h|_{[V_h]}$ being trivial (that is, $\rho^h_A(V_h)=V_h$ for any $A$). 
	\end{proposition}

\begin{proof}
From the discussion above it is clear that the only elements of 
$\mathcal N_{b,h}^{\chi}$ left fixed by
	 $\rho^h$ are precisely those which are constant along each horospherical leaf. Hence, the proof essentially amounts to checking that the only static potentials of $(\mathbb H^n,b)$ meeting this property are proportional to $V_h$.  
This claim may be readily verified if we pass to the Poincar\'e half-space model 	$(\mathbb R^n_+,\widetilde b)$ of hyperbolic $n$-space, where $\mathbb R^n_+=\{z=(z_1,z_2,\cdots,z_n)\in\mathbb R^n;z_1>0\}$ and $ \widetilde b=z_1^{-2}\delta$. We assume here that the distinguished point $o\in\mathbb S^{n-1}_\infty$ is mapped onto the point at infinity $\{z_1=+\infty\}$, so that the horospheres above correspond to the hyperplanes in the family $\{z_1=c\}_{c>0}$ and the associated parabolic group $\mathcal P$ acts on  the hyperplane at infinity $\mathbb R^{n-1}_\infty=\{z\in\mathbb R^n;z_1=0\}$ by similitudes. As expected, when acting on $(\mathbb R^n_+,\widetilde b)$ by isometries, $\mathcal P^h$ preserves each horospherical leaf so that the only coordinate function left fixed  is $z_1$. On the other hand, it is known that the isometry relating $\widehat b$ and $\widetilde b$ satisfies
	\begin{equation}\label{direct:map}
	z_1=\frac{1-|y|^2_\delta}{|\partial_{y_1}-y|_\delta^2}, \quad \partial_{y_1}=(1,0,\cdots,0), 
	\quad 
	z_j=\frac{2y_j}{{|\partial_{y_1}-y|_\delta^2}}, \quad j\geq 2. 
	\end{equation}
We may combine this with (\ref{exp:stat:ball}) to check that  
\begin{equation}\label{stat:upper}
V_h(z)=z_1^{-1}, \quad  V_{(j)}(z)=z_jz_1^{-1},\quad j\geq 2,
\end{equation}
 which finishes the proof. 
\end{proof}

Henceforth we drop the symbol $\chi$ from the notation, so the corresponding static manifold will be denoted by $(\mathbb H^n_h,b,\Sigma_h)$, etc. Alternatively, we may take $\chi=1$ for simplicity.
Proceeding
in analogy with Definition \ref{def:as:hyp}, we may consider an asymptotically hyperbolic manifold, say $(M,g,\Sigma)$, with a non-compact boundary $\Sigma$ and modeled  at infinity on the static manifold $(\mathbb H^n_h,b,\Sigma_h)$; compare with \cite[Definition 1]{chai2021asymptotically}. Also, as in Theorem \ref{finitemass} and with a self-explanatory notation,  we may  define a mass invariant for $(M,g,\Sigma)$ by setting 
\begin{equation}\label{def:mass:par}
\mathfrak m_{h,F}(V_h)=\lim_{r\to
	+\infty}\left[\int_{S^{n-1}_{r,+}}\langle\mathbb U(V_h,e),\mu\rangle
dS^{n-1}_{r,+}-
\int_{S^{n-2}_{r}}V_he(\eta,\vartheta)dS^{n-2}_{r}\right],
\end{equation} 
where as usual $F$ is a chart at infinity. 
Of course, here we rely on the decomposition (\ref{sp:irr}) to restrict the mass functional to $[V_h]$ and disregard its action on the complement $[V_{(2)},\cdots,V_{(n)}]$ of $[V_h]$ in $\mathcal N_{b,h}$. However, we argue in Remark \ref{ham:charge} below that the (vector-valued) invariant which arises by restricting the mass functional to this complement qualifies  as a `'center of mass'' of the underlying manifold. Also, we point out that the model $\mathbb H^n_h$ is compactified by adding a single point at infinity,  so the picture here is quite similar to the definition of the mass invariant in \cite{almaraz2014positive}.
The important point now is that an analogue of Lemma \ref{rigid} guarantees that two charts at infinity differ by an element of $\mathcal P^h$ (up to a term that vanishes as $r\to +\infty$). Proceeding as in the proof of Theorem \ref{geoinv} and taking Proposition \ref{action:parab} into account, we thus see that the right-hand side in (\ref{def:mass:par}) actually does {\em not} depend on which chart is used and the conclusion is that the mass 
\[
\mathfrak m_h:=\mathfrak m_{h,F}(V_h)
\] 
is a {\em numerical} invariant. 
 		
In order to obtain a Witten-type formula for $\mathfrak m_h$ in the spin category under suitable dominant energy conditions, we must impose an appropriate boundary condition on spinors along $\Sigma$. In view of Remark \ref{theta:v}, it is natural to implement this by considering the MIT bag boundary operator $Q_{h}={\bf i}\c(\nu):\Gamma(\mathbb SM|_\Sigma)\to \Gamma(\mathbb SM|_\Sigma)$, where $\nu$ is the inward unit normal along $\Sigma$. 
Notice that, differently from the $\theta$-boundary operator in (\ref{mu:bd}) which involves the chirality operator $Q=\omega$ and therefore only makes sense for $n$ even, $Q_h$ is well defined in {\em any} dimension.
In any case, spinors are required to satisfy $Q_{h}\Psi=\pm\Psi$ along $\Sigma$. Now we must check that there exist plenty of such spinors in the model space. 

\begin{proposition}\label{plenty}
	If $n=2k$ or $n=2k+1$ then, with self-explanatory notation,
		\[
		\dim_{\mathbb C}\mathcal K^{b,\pm,(\pm)_{h}}(\mathbb S\mathbb H^n_h)=2^{k-1}
		\]
	In particular, $\mathcal K^{b,\pm,(\pm)_{h}}(\mathbb S\mathbb H^n_h)\neq\{0\}$.	
		\end{proposition} 

\begin{proof}
	Just repeat the proof of Proposition \ref{dimkill} with $\theta=\pi/2$, which, as already remarked,  morally corresponds to the case considered here. 
\end{proof}

We now confirm that the well-known link between imaginary Killing spinors and static potentials works fine here as well; compare with Proposition \ref{link:stop}.

	\begin{proposition}\label{link:stop:2}
	For any $\Phi\in\mathcal K^{b,\pm,(\pm)_{h}}(\mathbb H^n_h)$ we have that $V_\Phi:=|\Phi|^2\in [V_h]$. In fact, if $\Phi=\Phi_{u,\pm}$ for some $\nabla^\delta$-parallel spinor $u$ as in (\ref{presc}) then $V_\Phi=|u|_\delta^2V_h$.  
\end{proposition}

\begin{proof}
	We take $\Phi=\Phi_{u,\pm}$ and note that the derivation of  (\ref{ref:need}) only uses the corresponding boundary condition. Thus, if we use it with $\tau=1$ we get ${\bf i}\langle \c^b(\nu)\Phi,\Phi\rangle=\pm|\Phi|^2$ along $\Sigma_h$, where here $\nu=-\eta_h$.  Bearing in mind that $\nabla^\pm\Phi=0$ and comparing with 
	(\ref{re:im}) we find that $V_\Phi=-\nu V_\Phi$ .  
By bringing back the parameter $\chi>0$ and letting it vary, we see that
this holds true everywhere along $\mathbb H^n$ if we think of $\nu$ as being the  (globally defined) unit normal along the leaves of the horospherical foliation. 
To investigate the constraints imposed on $V_\Phi$ by this equation, we  
	pass to the half-space model appearing in the proof of Proposition \ref{action:parab}. 
	We may invert (\ref{direct:map}) to obtain
	\[
	y_1=\frac{|z|_\delta^2-1}{|z+\partial_{z_1}|_\delta^2},\quad y_j=\frac{2z_j}{|z+\partial_{z_1}|_\delta^2}, \quad j\geq 2,
	\]
	and then compute that 
	\[
	V_{(1)}(z)=\frac{2y_1}{1-|y|_\delta^2}=\frac{2(|z|_\delta^2-1)|z+\partial_{z_1}|_\delta^2}{|z+\partial_{z_1}|_\delta^4-(|z|_\delta^2-1)^2-4(|z|_\delta^2-z_1^2)}=\frac{|z|_\delta^2-1}{2z_1}.
	\]
	Combining this with (\ref{stat:upper}) we see that, in this half-space model, (\ref{exp:ext:im:s}) may be rewritten as 
	\begin{eqnarray*}
	V_\Phi(z)
	& = & |u|^2_\delta z_1^{-1}+\frac{1}{2}\left({\bf i}\langle \c^\delta(\partial_{y_1})u,u\rangle_\delta+|u|_\delta^2\right)(|z|_\delta^2-1)z_1^{-1}\\
	& & \quad 
	\mp{\bf i}\sum_{j\geq 2}\langle \c^\delta(\partial_{y_j})u,u\rangle_\delta z_jz_1^{-1}. 
	\end{eqnarray*}
	Since $\nu=z_1\partial_{z_1}$ we get 
	\begin{eqnarray*}
	(\nu V_\Phi)(z)
	& = & -|u|^2_\delta z_1^{-1}+\frac{1}{2}\left({\bf i}\langle \c^\delta(\partial_{y_1})u,u\rangle_\delta+|u|_\delta^2\right)(2-(|z|^2_\delta-1)z_1^{-1})\\
	& & \quad 
	\pm{\bf i}\sum_{j\geq 2}\langle \c^\delta(\partial_{y_j})u,u\rangle_\delta z_jz_1^{-1}, 
\end{eqnarray*}
which gives 
	\[
	 {\bf i}\langle \c^\delta(\partial_{y_1})u,u\rangle_\delta+|u|_\delta^2=
	 	(\nu V_\Phi+V_\Phi)(z)=0. 
\]
Thus, if we set $\langle \c^\delta(\partial_{y_j})u,u\rangle=a_j{\bf i}$, $a_j\in\mathbb R$, it follows from this and (\ref{type:I}) that  
\[
\sum_{j\geq 2}a_j^2=-\sum_{j\geq 2}\langle \c^\delta(\partial_{y_j})u,u\rangle^2=|u|_\delta^4+
\langle\c^\delta(\partial_{y_1})u,u\rangle^2=0.
\]
Hence, $V_\Phi(z) = |u|^2_\delta z_1^{-1}$
which completes the proof in view of (\ref{stat:upper}).
\end{proof}
 		
We may now outline the argument leading to a positive mass inequality for $\mathfrak m_h$, with a rigidity statement included. As before, we start with some  $\Phi\in\mathcal K^{b,\pm,(\pm)_{h}}(\mathbb H^n_s)$. We may assume that $\Phi=\Phi_{u,\pm}$, where $\nabla^\delta u=0$ and $|u|_\delta=1$, so that $V_{\Phi}=V_h$ by Proposition \ref{link:stop:2}.  Under the appropriate dominant energy conditions as in Theorem \ref{main:th:horo} below, the usual analytical machinery may be employed to obtain another spinor $\Psi_u\in\Gamma(\mathbb SM)$ which is Killing harmonic, satisfies the corresponding MIT bag boundary condition along $\Sigma$ and asymptotes $\Phi_{u,\pm}$ at infinity. A standard computation provides the corresponding Witten-type formula:
	\begin{eqnarray}\label{witten:horo}
	\frac{1}{4}\mathfrak m_h & = & \int_M\left(|\nabla^{\pm}\Psi_\Phi|^2+\frac{R_g+n(n-1)}{4}|\Psi_u|^2\right)dM\\
	& & \quad +\frac{1}{2}\int_\Sigma (H_g-(n-1))|\Psi_u|^2d\Sigma.\nonumber
\end{eqnarray}
This is the key ingredient in establishing the following positive mass theorem.

	\begin{theorem}\label{main:th:horo}
	Let $(M,g,\Sigma)$ be an asymptotically hyperbolic spin manifold (modeled at infinity on $(\mathbb H^n_h,b,\Sigma_h)$) with $R_g\geq -n(n-1)$ and $H_g\geq (n-1)$. Then 
	$\mathfrak m_h\geq 0$ and the equality holds if and only if 
$(M,g,\Sigma)$ is isometric to 
 $(\mathbb H^n_h,b,\Sigma_h)$.
\end{theorem}

\begin{proof}
That $\mathfrak m_h\geq 0$ already follows from (\ref{witten:horo}). As for the rigidity statement, if $\mathfrak m_h=0$ then again by (\ref{witten:horo}), this time combined with Propositions \ref{plenty} and \ref{link:stop:2}, $(M,g,\Sigma)$ carries as many imaginary Killing spinors satisfying a MIT bag boundary condition as $(\mathbb H^n_h,b,\Sigma_h)$ does, which implies that $g$ is locally hyperbolic, 
$(\Sigma,\gamma)\hookrightarrow (M,g)$, $\gamma=g|_{\Sigma}$, is totally umbilical (with $H_g=(n-1)$) and
$\gamma$ is flat. It follows that $(\Sigma,\gamma)=(\mathbb R^{n-1},\delta)$ isometrically, so we may glue together  $(M,g,\Sigma)$ and the complement of the corresponding horoball model along the common boundary to conclude, again appealing to \cite{herzlich2016computing,chrusciel2003mass}, that $(M,g,\Sigma)=(\mathbb H^n_h,b,\Sigma_h)$ isometrically.
\end{proof}

 Clearly, the rigidity statement here covers the horospherical case in \cite[Theorem 2]{souam2021mean}; see also the companion result in Remark \ref{ess:pmt} below. We emphasize, however, that it actually provides a much sharper non-deformability result for the embedding $\Sigma_h\hookrightarrow \mathbb H^n$; compare with Remark \ref{ext:best}. 

\begin{remark}\label{ess:pmt}
Essentially the same argument as above yields a mass formula similar to (\ref{witten:horo}) for asymptotically hyperbolic manifolds modeled at infinity on the {\em complement} of a horoball $\mathbb H^n_h$ in $\mathbb H^n$, the only difference being that $H_g-(n-1)$ gets replaced by $H_g+(n-1)$ in the boundary integral. This corresponds to the case $\theta=-\pi/2$ ($\tau=-1$) in the notation of the previous sections, which means that we must use $-{\bf i}\c(\nu)$ as the associated boundary operator. Perhaps the most noticeable modification in the argument occurs while repeating  the proof of Proposition \ref{link:stop:2} with $V_\Phi=\nu'V_\Phi$, where  $\nu'=-\nu=-z_1\partial_{z_1}$ is the appropriate unit normal along the horospherical foliation, 
so we end up with the same conclusion.
\end{remark}

\begin{remark}\label{limit:th}
	The definition (\ref{def:mass:par}) suggests that we should think of $\mathfrak m_h$ as a multiple of $V_h$. On the other hand, we may equip $\mathcal N_{b,h}$ with a {\em degenerate} symmetric bi-linear form, say $\langle\,,\rangle_h$, by declaring that
	$V_h$ is {\em null} (that is, $\langle V_h,V\rangle_h=0$ for any $V\in\mathcal N_{b,h}$) and that $\langle V_{(j)},V_{(j')}\rangle_h=-\delta_{jj'}$, $j,j'\geq 2$, so that  $\mathcal P^h$ acts {\em isometrically} on $(\mathcal N_{b,h},\langle\,,\rangle_h)$ via $\rho^h$. Intuitively, we may view $(\mathcal N_{b,h},\langle\,,\rangle_h)$ as obtained from $(\mathcal N_{b,s},\langle\,,\rangle_s)$ by collapsing its isotropic cone into a single null line (namely, $[V_h]$), so that no time-like direction survives in the process. In terms of the isometry groups of the underlying static spaces, this amounts to passing from $O^\uparrow(1,n-1)$ to $O(n-1)\ltimes \mathbb R^{n-1}$, with this latter group acting on $([V_{(2)},\cdots,V_{(n)}],\langle\,,\rangle_h)=(\mathbb R^{n-1},-\delta)$ by isometries. Using this setup,
	we may interpret Theorem \ref{main:th:horo} as saying that, under the given dominant energy conditions, $\mathfrak m_h$ is null-positive (that is, a positive multiple of $V_h$) unless it vanishes, in which case the underlying manifold is isometric to the model. In this way, Theorem \ref{main:th:horo} may  be regarded as the natural rewording of  Theorem  \ref{main:pmt} as $\theta\to\pi/2$. 
	\end{remark}

\begin{remark}\label{ham:charge}
	The splitting (\ref{sp:irr}) of $\mathcal N^{h}_b$ under the natural $\mathcal P^h$-action suggests defining a complementary (vector-valued) Hamiltonian charge  for any asymptotically hyperbolic manifold modeled at infinity on $\mathbb H^n_h$ (or on its complement) by simply restricting the mass functional to $[V_{(2)},\cdots,V_{(n)}]$. More precisely, in the presence of a chart at infinity $F$, a vector $\mathcal C_{h,F}\in [V_{(2)},\cdots,V_{(n)}]$ is defined by 
	\[
	\langle \mathcal C_{h,F}^a,V_{(a)}\rangle_h=\lim_{ r\to
		+\infty}\left[\int_{S^{n-1}_{ r,+}}\langle\mathbb U(V_{(a)},e),\mu\rangle
	dS^{n-1}_{ r,+}-
	\int_{S^{n-2}_{ r}}V_{(a)}e(\eta,\vartheta)dS^{n-2}_{ r}\right];
	\]
	compare with (\ref{def:mass:par}).
	As expected, $\mathcal C_{h,F}$ transforms as a vector in the representation $\rho^h$ as we pass from one chart at infinity to another. 
	It would be interesting to investigate the basic properties of this invariant and how it relates to the asymptotic geometry of the underlying  manifold. In this regard, we note that \cite[Theorem 2]{chai2021asymptotically} provides an Ashtekar-Hansen-type formula for this invariant (and for our mass $\mathfrak m_h$ as well), very much in the spirit of results found in \cite{de2019mass,chai2022evaluation} for the case $s=0$. Specifically, in $\mathbb R^{1,n}$ consider the vector fields
	\[
	X=C+\langle C ,x\rangle_{1,n}x-Y_{01}
	\] 
	and 
	\[
	X_a=C_a+\langle C_a ,x\rangle_{1,n}x-Y_{01}-Y_{0a},\quad 2\leq a\leq n, 
	\]
	where 
	\[
	C=\partial_{x_0}+\partial_{x_1}, \quad C_a=C+\partial_{a}, 
	\]
	and
	\[
	Y_{01}=x_0\partial_{x_1}+x_1\partial_{x_0}, \quad Y_{0a}=x_0\partial_{x_a}+x_a\partial_{x_0}. 
	\]
Upon restriction to $\mathbb H^n$, $X$ and $X_a$ are conformal fields which remain tangent to $\Sigma_h=\{x_0=x_1\}$. 
	With this notation, and under the obvious identifications induced by $F$, Chai's formulae read as 
	 	\begin{equation}\label{chai:1}
	 \mathfrak m_h=c_n\lim_{ r\to
	 	+\infty}\left[\int_{\mathcal S^{n-1}_{ r,+}}\widetilde E_g(X,\mu)
	 d\mathcal S^{n-1}_{ r,+}-
	 \int_{\mathcal S^{n-2}_{ r}}\widetilde\Pi_g(X,\vartheta)d\mathcal S^{n-2}_{ r}\right],
	 \end{equation}
	 and 
	\begin{equation}\label{chai:2}
	\mathcal C_{h,F}^a=c_n'\lim_{ r\to
		+\infty}\left[\int_{\mathcal S^{n-1}_{ r,+}}\widetilde E_g(X_a,\mu)
	d\mathcal S^{n-1}_{ r,+}-
	\int_{\mathcal S^{n-2}_{ r}}\widetilde\Pi_g(X_a,\vartheta)d\mathcal S^{n-2}_{ r}\right],
	\end{equation}
	where 
	\[
	\widetilde E_g={\rm Ric}_g-\frac{R_g}{2}g-\frac{(n-1)(n-2)}{2}g
	\]
	is the modified Einstein tensor of $g$ and
	\[
	\widetilde \Pi_g=\Pi_g-H_g\gamma+({n-2})\gamma
	\]
	is the modified Newton tensor of the embedding $(\Sigma,\gamma)\hookrightarrow (M,g)$. Also,
$\mathcal S^{n-2}_r=\partial \mathcal S^{n-1}_{r,+}$ and, in the half-space model above, we take $\mathcal S^{n-1}_{r,+}$ to be the Euclidean hemisphere of radius $r>1$ centered at the ``origin'' $(1,0,0,\cdots,0)$ and lying above the horosphere $\Sigma_h=\{z_1=1\}$, so that $\mathcal S^{n-1}_{r,+}$ is a piece of an equidistant hypersurface with mean curvature $H_{(r)}=(n-1)/r$.	
This	shows that $\mathcal C_h$ qualifies as a kind of `'center of mass'' for the given asymptotically hyperbolic manifold. 
	In any case, at least in dimension $n=3$, $\mathcal C_h$ should play a role in solving the {\em relative} isoperimetric problem for large values of the enclosed volume in such manifolds; compare with \cite[Theorem 2.28]{almaraz2020center}. 
\end{remark}

		\bibliographystyle{alpha}
		\bibliography{mass-hyp-bd-tu}

	\end{document}